\documentclass[11pt]{amsart}
\allowdisplaybreaks
\usepackage[english]{babel}
\usepackage[pdftex,textwidth=400pt,marginratio=1:1]{geometry}
\usepackage{amsfonts}
\usepackage[dvips]{graphics}
\usepackage[colorinlistoftodos]{todonotes}
\usepackage{amsmath}
\usepackage{combelow} 
\usepackage{amsthm}
\usepackage{amssymb}
\usepackage{bbm}
\usepackage{cancel}
\usepackage{color}
\usepackage[curve]{xypic}
\usepackage{graphicx}
\usepackage{mathabx}
\usepackage{tikz-cd}
\usepackage{tikz}
\usetikzlibrary{decorations.markings}

\usepackage{epigraph} 
\usepackage{hyperref}
\usepackage{stmaryrd}
\usepackage{enumitem}
\usepackage{soul}
\usepackage[mathscr]{eucal}
\usepackage{etoolbox}
\newtheorem{theorem}{Theorem}[section]
\newtheorem{corollary}[theorem]{Corollary}
\newtheorem{proposition}[theorem]{Proposition}
\newtheorem{lemma}[theorem]{Lemma}

\theoremstyle{definition}

\newtheorem{remark}[theorem]{Remark}

\theoremstyle{property}

\DeclareFontFamily{OT1}{rsfs}{}
\DeclareFontShape{OT1}{rsfs}{n}{it}{<-> rsfs10}{}
\DeclareMathAlphabet{\curly}{OT1}{rsfs}{n}{it}

\newcommand\I{\mathcal I}

\renewcommand\L{\mathcal L}

\renewcommand\O{\mathcal O}
\newcommand\PP{\mathbb P}

\newcommand\EE{\mathbb E}

\newcommand\T{\mathbb T}
\newcommand\E{\mathbb E}
\newcommand\C{\mathbb C}

\newcommand\FF{\mathbb F}

\newcommand\sfZ{\mathsf Z}

\newcommand\Q{\mathbb Q}

\newcommand\Z{\mathbb Z}

\newcommand\SU{\mathrm{SU}}

\newcommand\vd{\mathrm{vd}}
\newcommand\pt{\mathrm{pt}}
\newcommand\vir{\mathrm{vir}}

\newcommand\GL{\mathrm{GL}}
\newcommand\bfbeta{{\boldsymbol{\beta}}}
\newcommand\bfn{{\boldsymbol{n}}}
\newcommand\bfa{{\boldsymbol{a}}}

\newcommand\SW{\mathrm{SW}}

\newcommand\rk{\operatorname{rk}}

\newcommand\tr{\operatorname{tr}}

\newcommand\Hom{\operatorname{Hom}}
\renewcommand\hom{\mathcal{H}{\it{om}}}

\newcommand\Spec{\operatorname{Spec}\,}
\newcommand\Hilb{\operatorname{Hilb}}

\newcommand\INTO{\ar@{^{(}->}[r]}

\DeclareRobustCommand{\SkipTocEntry}[4]{}

\def\Gr{\mathbf{Gr}}
\def\Pic{\mathrm{Pic}}
\def\GL{\mathrm{GL}}

\def\TT{\mathbb{T}}

\def\DR{\mathrm{DR}}
\def\bfell{\boldsymbol{\ell}}

\def\and{\quad\mathrm{and}\quad}
\newcommand{\MHM}{\mathrm{MHM}}

\begin{document}

\author[N.~Arbesfeld, M.~Kool and T.~Laarakker]{Noah~Arbesfeld, Martijn~Kool and Ties~Laarakker}

\title[Vafa-Witten invariants from framed sheaves]{Vafa-Witten invariants from wall-crossing for framed sheaves}

\begin{abstract}
We consider the refined $\SU(r)$ Vafa-Witten partition function of a smooth projective surface with non-zero holomorphic 2-form. This partition function has a vertical contribution, expressible in terms of nested Hilbert schemes.

First, we write the vertical contribution in terms of $\chi_y$-genera of moduli spaces of framed sheaves on $\PP^2$. 

Then, we state two wall-crossing identities for moduli spaces of framed sheaves: a blow-up formula due to Kuhn-Leigh-Tanaka and a new stable/co-stable wall-crossing formula. We prove the latter using the theory of mixed Hodge modules.

We apply these identities to obtain constraints on Vafa-Witten invariants predicted by conjectures of G\"ottsche and the second- and third-named authors. For $r=2$, we obtain a proof of the vertical part of a celebrated formula by Vafa-Witten.
\end{abstract}
\maketitle

\section{Introduction}\label{sec:Intro}

In 1994, Vafa-Witten \cite{VW} studied $S$-duality of certain partition functions arising from a topological sector of $N=4$ supersymmetric Yang-Mills theory on a 4-manifold. When the 4-manifold is a complex smooth projective surface $S$, a mathematical definition of the $\SU(r)$ Vafa-Witten partition function was proposed in 2017 by Tanaka-Thomas \cite{TT1}. See also \cite{GSY2} for a discussion from the perspective of reduced Donaldson-Thomas invariants.

Let $(S,H)$ be a smooth polarized surface.
Let $r \in \Z_{>0}$ and let $L$ be a line bundle on $S$. A Higgs pair $(E,\phi)$ on $S$ consists of a torsion free sheaf $E$ on $S$ and a morphism $\phi \colon E \to E \otimes K_S$ satisfying $\tr(\phi) = 0$.
For any $c_2 \in H^4(S,\Z) \cong \Z$, denote by
$$N:=N_S^H(r,L,c_2)$$
the moduli space of rank $r$ $H$-stable Higgs pairs $(E,\phi)$ on $S$ with $\det(E) \cong L$ and $c_2(E) = c_2$. Tanaka-Thomas prove in \cite{TT1} that
$N$ admits a symmetric perfect obstruction theory. Denote the virtual tangent bundle by $T_N^{\vir}$ and its dual by $\Omega_N^{\vir}$. The non-compact moduli space $N$ has a $\C^*$-action scaling the Higgs field.

Assume there are no rank $r$ strictly $H$-semistable Higgs pairs $(E,\phi)$ on $S$ satisfying $\det(E) \cong L$ and $c_2(E) = c_2$. Then the fixed locus $N^{\C^*}$ is compact and can be decomposed
into ``open and closed'' components 
$$
N^{\C^*} = \bigsqcup_{\mu} N^{\C^*}_{\mu},
$$
where $\mu$ runs over all sequences of positive integers $\mu = (\mu_0, \ldots, \mu_{\ell})$, for arbitrary $\ell \geq 0$, satisfying $\mu_0+\cdots+\mu_\ell=r$. The component $N^{\C^*}_\mu$ contains the $\C^*$-fixed Higgs pairs $(E,\phi)$ with weight decomposition
$$
E = \bigoplus_{i=0}^{\ell} E_i \otimes y^{-i}, \quad \rk(E_i) = \mu_i,
$$
where $y$ denotes a primitive character of the $\C^*$-action. 

Consider the sequences $\mu=(r)$ and $\mu = (1,\ldots, 1)=:(1^r)$. The components $N^{\C^*}_{(r)}$ and $N^{\C^*}_{(1^r)}$ are called the \emph{horizontal} and \emph{vertical} component of $N^{\C^*}$ respectively. The horizontal component $N^{\C^*}_{(r)}$ is isomorphic to the Gieseker-Maruyama-Simpson moduli space
$M:=M_S^H(r,L,c_2)$
of rank $r$ $H$-stable torsion free sheaves $E$ on $S$ with $\det(E) \cong L$ and $c_2(E) = c_2$ \cite{HL}.

Denote by $[N]^{\vir} \in H_*^{\C^*}(N)$ the virtual cycle and by $\O_N^{\vir} \in K_0^{\C^*}(N)$ the virtual structure sheaf of $N$, as defined in \cite{BF, CFK}. 
Write $K^{\vir}_{N} = \det(\Omega_N^{\vir})$. When $K^{\vir}_{N}$ admits a square root $(K^{\vir}_{N})^{\frac{1}{2}}$, the twisted virtual structure sheaf is defined by $\widehat{\O}_N^{\vir} = \O_N^{\vir} \otimes (K^{\vir}_{N})^{\frac{1}{2}}$ as in  \cite[Definition 3.1]{NO}. Consider the invariants
\begin{equation} \label{def1}
\int_{[N]^{\vir}} 1, \quad \chi(N,\widehat{\O}^{\vir}_N).
\end{equation}
Since $N$ is non-compact, these expressions are defined by the virtual localization formulae \cite{GP}
\begin{align}  \label{def2}
\int_{[N^{\C^*}]^{\vir}} \frac{1}{e(\nu_N^{\vir})} \in \Q, \quad \chi\Bigg(N^{\C^*},\frac{\O_{N^{\C^*}}^{\vir} \otimes (K^{\vir}_{N})^{\frac{1}{2}}|_{N^{\C^*}}}{\Lambda_{-1} (\nu_N^{\vir})^\vee }\Bigg) \in \Q(y^{\frac{1}{2}}), 
\end{align}
where $\nu_N^{\vir}$ denotes the virtual normal bundle, $e(-)$ the equivariant Euler class, and $y = e^t$ for $t = c_1^{\C^*}(y)$. The first invariants were introduced by Tanaka-Thomas \cite{TT1} and the second (K-theoretic) invariants by Thomas \cite{Tho}. The existence of the square root $(K^{\vir}_{N})^{\frac{1}{2}}|_{N^{\C^*}}$ on the fixed locus appearing in \eqref{def2} is shown by  \cite[Proposition 2.6]{Tho}. 

The $\C^*$-fixed part of $T_N^{\vir}|_{M}$ is equal to the virtual tangent bundle $T_M^{\vir}$ of the natural perfect obstruction theory on $M$ previously studied by Mochizuki \cite{Moc} and, point-wise, is given by
\begin{align*}
T_M^{\vir}|_{[E]} \cong R\Hom_S(E,E)_0[1], 
\end{align*}
where $(-)_0$ denotes the trace-free part. Its rank is given by
$$
\vd := \vd(r,L,c_2) :=  2rc_2 - (r-1)c_1(L)^2 - (r^2-1) \chi(\O_S).
$$
Then the contribution of $M$ to \eqref{def1} equals (up to the sign $(-1)^{\vd} $)
\begin{equation} \label{hor}
e^{\vir}(M), \quad  \widehat{\chi}_{-y}^{\vir}(M) :=  y^{-\tfrac{\vd}{2}} \chi_{-y}^{\vir}(M), 
\end{equation}
where $e^{\vir}(M)$ and $\chi_{y}^{\vir}(M)$ are the virtual Euler characteristic and $\chi_y$-genus defined by Fantechi-G\"ottsche \cite{FG}. The symmetrized virtual $\chi_y$-genus $\widehat{\chi}_{-y}^{\vir}(M)$ is invariant under $y \mapsto 1/y$. If $H K_S < 0$ or $K_S \cong \O_S$, then only the horizontal components contribute and they are smooth of expected dimension. In this case, the virtual invariants \eqref{hor} reduce to (classical) Euler characteristics and Hirzebruch $\chi_y$-genera \cite[Proposition 7.4]{TT1}. 

We are interested in the case where $S$ has a non-zero holomorphic 2-form, i.e., $p_g(S)>0$. Then $M$ is typically singular and the vertical contribution to \eqref{def1} is non-zero.
For fixed $S,H,r,L$ such that $H_1(S,\Z)_{\mathrm{tor}} = 0$,\footnote{In the case where $H_1(S,\Z)$ has torsion, the generating series has to be modified \cite{Wit}.} the $\SU(r)$ Vafa-Witten partition function is defined by\footnote{The factor in front of the sum is needed for good modular behaviour.}
$$
\sfZ_{S,H,L}^{\SU(r)}:= r^{b_1(S)-1} q^{-\frac{\chi(\O_S)}{2r} + \frac{r K_S^2}{24}} \sum_{c_2} (-1)^{\vd(r,L,c_2)} q^{\frac{\vd(r,L,c_2)}{2r}} \chi(N_S^H(r,L,c_2), \widehat{\O}^{\vir}),
$$
where $b_1(S)$ denotes the first Betti number of $S$. The partition function decomposes into contributions from the components $N_\mu^{\C^*}$ as 
\begin{equation} \label{def:genfundeco}
\sfZ_{S,H,L}^{\SU(r)} = r^{b_1(S)-1} \sum_\mu \sfZ_{S,H,L}^{\mu}.
\end{equation}

Suppose $p_g(S)>0$. Picking a non-zero holomorphic 2-form $\theta \in H^0(S,K_S)$, Thomas constructs a cosection of the $\C^*$-fixed obstruction theory on $N^{\C^*}$. He uses the cosection to prove \cite[Thm.~5.23]{Tho}, that, for $r$ prime, only the horizontal and vertical components have nonzero contributions to the partition function. So,
$$
\sfZ_{S,H,L}^{\SU(r)} = r^{b_1(S)-1} \big(\sfZ_{S,H,L}^{(r)} + \sfZ_{S,H,L}^{(1^r)}\big).
$$

For arbitrary $S,H,r,L$, there may be strictly $H$-semistable Higgs pairs and one should replace $N_S^H(r,L,c_2)$ by the moduli space of Joyce-Song Higgs pairs $P_S^H(r,L,c_2)$ \cite{TT2, Tho}. These spaces carry a $\C^*$-action and components of fixed loci are indexed by $\mu$. 
The $\SU(r)$ Vafa-Witten partition function can be defined similarly. It is also of the form \eqref{def1} and, for $r$ prime, again only the horizontal and vertical components contribute \cite{Tho}. The definition of the invariants in the semistable case depends on a conjecture \cite[Conjecture 5.2]{Tho} proved by Liu \cite[Theorem 1.4]{Liu2}.

\subsection{Vertical contribution}

Let $(S,H)$ be a smooth polarized surface satisfying $H_1(S,\Z) = 0$ and $p_g(S)>0$. 
The Higgs pairs in the vertical components decompose into rank 1 eigensheaves. A rank 1 torsion free sheaf on $S$ is of the form $I_Z \otimes L$, where $Z \subset S$ is 0-dimensional subscheme with ideal sheaf $I_Z \subset \O_S$ and $L$ is a line bundle on $S$. We now discuss how the vertical components can be realized as certain nested Hilbert schemes of points and curves on $S$. 

We denote by $\Hilb^n(S)$ the Hilbert scheme of $n$ points on $S$ and, for an algebraic class $\beta \in H^2(S,\Z)$, we write $|\beta|$ for the linear system of effective divisors on $S$ with class $\beta$. 
For any $\bfn = (n_0,\ldots, n_{r-1}) \in \Z_{\geq 0}^r$ and effective classes $\bfbeta = (\beta_1, \ldots, \beta_{r-1}) \in H^2(S,\Z)^{r-1}$, we define 
\begin{align*}
\Hilb^{\bfn}(S) :=  \prod_{i=0}^{r-1} \Hilb^{n_i}(S), \quad |\bfbeta| := \prod_{i=1}^{r-1} |\beta_i|.
\end{align*}
Gholampour-Thomas \cite{GT1, GT2} introduce an incidence subscheme of $\Hilb^{\bfn}(S) \times |\bfbeta|$ (see also \cite{GSY1})
\begin{align*}
\Hilb_{\bfbeta}^{\bfn}(S) = \Big\{(Z_0, \ldots, Z_{r-1},C_1, \ldots, C_{r-1}) \, : \, I_{Z_{i-1}}(-C_i) \subset I_{Z_i} \quad   \textrm{for all } i  \Big\}.
\end{align*}
The union of vertical components
$$
\bigsqcup_{c_2} N_S^H(r,c_1,c_2)_{(1^r)}^{\C^*}
$$ 
is isomorphic to a union of  $\Hilb_{\bfbeta}^{\bfn}(S)$ for certain $\bfn$, $\bfbeta$. In Section \ref{sec:GT}, we give an expression for these $\bfn$, $\bfbeta$, in terms of $r,c_1,c_2$.

Roughly speaking, the scheme $\Hilb_{\bfbeta}^{\bfn}(S)$ can be realized as the degeneracy locus of a morphism of vector bundles on the smooth ambient variety $\Hilb^{\bfn}(S) \times |\bfbeta|$  \cite{GT1,GT2}. Gholampour-Thomas show that this description can be used to endow $\Hilb_{\bfbeta}^{\bfn}(S)$ with a perfect obstruction theory. Crucially, the virtual cycle from the ``degeneracy'' perfect obstruction theory of Gholampour-Thomas coincides with the virtual cycle from the ``$\C^*$-localized'' perfect obstruction theory of Tanaka-Thomas  (as we review in Section \ref{sec:GT}). The former is much easier to calculate.

Using these results, the third-named author derived in \cite{Laa1, Laa2} a universality statement for the vertical part of the $\SU(r)$ Vafa-Witten partition function, which we now describe. For any $a,b \in H^2(S,\Z)$, define
\begin{align} \label{def:delta}
\delta_{a,b} := \left\{\begin{array}{cc} 1 & \mathrm{if \, } a-b \in rH^2(S,\mathbb{Z}) \\ 0 & \mathrm{otherwise.} \end{array} \right.
\end{align}
For $p_g(S)>0$, we denote the Seiberg-Witten invariant of $\beta \in H^2(S,\Z)$ by $\SW(\beta) \in \Z$. For an algebraic class $\beta \in H^2(S,\Z)$, the linear system $|\beta|$ has a perfect obstruction theory and virtual cycle $|\beta|^{\vir}$ in degree $\beta(\beta-K_S) / 2$. If $|\beta|^{\vir} \neq 0$, then $\beta^2 = \beta K_S$ and $\SW(\beta) = \deg(|\beta|^{\vir})$ \cite[Proposition 6.3.1]{Moc} (see also \cite{DKO} in greater generality).\footnote{Equality of algebro- and differential geometric Seiberg-Witten invariants was finalized in \cite{CK}.} We use the following Fourier expansions of the modular form $\Delta$ and the weak Jacobi form $\phi_{-2,1}$
\begin{align*}
\Delta(q) &:= q \prod_{n=1}^{\infty} (1-q^n)^{24}, \\ 
\phi_{-2,1}(q,y) &:= (y^{\frac{1}{2}} - y^{-\frac{1}{2}})^2 \prod_{n=1}^{\infty} \frac{(1- y q^n)^2 (1-y^{-1} q^n)^2}{(1-q^n)^4}.
\end{align*}

\begin{theorem}[Laarakker] \label{thm:Laarakker}
For any $r>1$, there exist universal generating series $A$, $B$, $\{C_{ij}\}_{1 \leq i \leq j \leq r-1}$ with the following property.\footnote{Up to an explicit normalization term, given in Section \ref{sec:GT},  the universal series $A,B,C_{ij}$ lie in $1+ q \, \Q(y^{\frac{1}{2}})[[q]]$. They only depend on $r$.} For any smooth polarized surface $(S,H)$ satisfying $H_1(S,\Z) = 0$, $p_g(S)>0$, and $L \in \Pic(S)$, we have
\begin{align*}
\frac{\sfZ_{S,H,L}^{(1^r)}}{(y^{\frac{1}{2}} - y^{-\frac{1}{2}})^{\chi(\O_S)}}  = A^{\chi(\O_S)} B^{K_S^2} \sum_{\bfa \in H^2(S,\Z)^{r-1}}  \delta_{c_1(L),\sum_i i a_i} \prod_{i} \SW(a_i) \prod_{i \leq j} C_{ij}^{a_i a_j}.
\end{align*}
\end{theorem}

\begin{remark}
Note that Theorem \ref{thm:Laarakker} also holds without the assumption $H_1(S,\Z) = 0$, in which case the right-hand side should be divided by the order of the group of $r$-torsion elements $\mathrm{Pic}(S)[r]$ \cite{Laa2}. 
Moreover, the theorem is valid in the presence of strictly $H$-semistable objects \cite{Laa2}. Nonetheless, in this greater generality, the universal functions $A, B, C_{ij}$ are unchanged. So, for the determination of the universal functions, we may always assume that $H_1(S,\Z) = 0$ and that there are no strictly $H$-semistable objects.
We observe that the right-hand side only depends on $c_1 := c_1(L)$, and we therefore write $\sfZ_{S,H,c_1}^{(1^r)}:=\sfZ_{S,H,L}^{(1^r)}$. 
\end{remark}

The universal series $A$ was determined for $y=1$ by Tanaka-Thomas in \cite[Theorem 1.7]{TT2}. The K-theoretic upgrade was found by Thomas \cite[Equation 1.3]{Tho} and Laarakker \cite[Theorem C]{Laa1}
\begin{equation} \label{eqn:A}
A = \frac{(-1)^{r-1}}{\phi_{-2,1}(q^r,y^r)^{\frac{1}{2}} \Delta(q^r)^{\frac{1}{2}}}.
\end{equation}

\subsection{Framed sheaves}\label{sec:intro2}

The departure point of the present paper is the observation that that the universal series appearing in Theorem \ref{thm:Laarakker} can be expressed in terms of integrals over moduli spaces of framed sheaves on $\PP^2$. These moduli spaces are fundamental examples of Nakajima quiver varieties. 
One introduction to their geometry can be found in \cite[Section 3.1]{NY1}. 

Denote by $M_{\PP^2}(r,n)$ the (fine) moduli space parametrizing isomorphism classes of rank $r$ framed sheaves $(E,\Phi)$ on $\PP^2$ satisfying with $c_2(E) = n$. That is to say, $E$ is a rank $r$ torsion free sheaf on $\PP^2$ with $c_2(E) = n$, which is locally free in a neighbourhood of the line $\ell_{\infty} = \{[0:x_1:x_2] : [x_1:x_2] \in \PP^1\}$, and
$$
\Phi \colon E|_{\ell_\infty} \stackrel{}{\rightarrow} \O_{\ell_\infty}^{\oplus r}
$$
is an isomorphism. The existence of the framing $\Phi$ implies that $c_1(E) = 0$. Then $M_{\PP^2}(r,n)$ is a smooth quasi-projective variety of dimension $2rn$. 

The action of $T_1 = (\C^*)^2$ on $\PP^2$ lifts to an action of $M_{\PP^2}(r,n)$. The torus $T_2 = (\C^*)^r$ acts on the framing by scaling the summands of $\O_{\ell_\infty}^{\oplus r}$. We endow $M_{\PP^2}(r,n)$ with a further trivial $T_3 = \C^*$-action.

Let $\TT = T_1 \times T_2 \times T_3$ and denote the corresponding equivariant parameters by
\begin{align*}
T_1: \quad t_1, t_2, \quad T_2: \quad e_0, \ldots,  e_{r-1}, \quad T_3: \quad y.
\end{align*}
Then we consider the (equivariant, symmetrized) $\chi_y$-genus of $M_{\PP^2}(r,n)$
\begin{align*}
\widehat{\chi}_{-y}(M_{\PP^2}(r,n)):=y^{-rn}\chi(M_{\PP^2}(r,n),\Lambda_{-y}\Omega_{M_{\PP^2}(r,n)}) &\in \Q(t_1,t_2,e_0,\ldots, e_{r-1},y).
\end{align*}
This expression is well-defined as a Laurent series by \cite[Lemma 4.2]{NY2} and can also be studied by equivariant localization (Section \ref{sec:framed}). The generating series \begin{align}\label{eq:adjSeries}\sum_n q^n \widehat{\chi}_{-y}(M_{\PP^2}(r,n))\in \Q(t_1,t_2,e_0,\ldots, e_{r-1},y)[[q]]\end{align} appears in the physics literature as (the instanton part of) the K-theoretic (or five-dimensional) $\mathrm{SU}(r)$ Nekrasov partition function with adjoint matter, an example of the Nekrasov partition functions introduced in \cite{Nek}.

We relate the universal functions of Theorem \ref{thm:Laarakker} to generating series of $\chi_y$-genera of $M_{\PP^2}(r,n)$. For any smooth projective surface $S$ satisfying $H_1(S,\Z) = 0$, we consider the following quadratic form
\begin{align} 
\begin{split} \label{def:Q}
Q \colon H^2(S,\Z)^{r-1} \rightarrow \Q, \quad Q(a_1,\ldots,a_{r-1}) = - \sum_{i<j} \frac{i(r-j)}{r} a_ia_j - \sum_{i=1}^{r-1} \frac{i(r-i)}{2r} a_i^2.
\end{split}
\end{align}
For a fixed $\bfa=(a_1,\ldots,a_{r-1}) \in H^2(S,\Z)^{r-1}$, we define line bundles $\{L_i\}_{i=0}^{r-1}$ on $S$ by 
\begin{equation} \label{def:Li}
a_i = c_1(L_{i-1} \otimes L_i^*), \quad i=1, \ldots, r-1.
\end{equation}
Note that this set of equations is underdetermined; the $L_i$ are only determined up to tensoring them by an overall line bundle $L$. However, all of our formulas will only depend on the ``differences'' $L_i^{*} \otimes L_j$.

When $S$ is a smooth projective toric surface with torus $(\C^*)^2 \cong T \subset S$, we denote the collection of maximal $T$-invariant affine open subsets covering $S$ by 
$
\{U_{\alpha} \cong \C^2\}_{\alpha=1}^{e(S)}.
$
Here each index $\alpha = 1, \ldots, e(S)$ corresponds to a 2-dimensional cone of the fan of $S$ \cite{Ful}. Each chart $U_{\alpha}$ has a unique $T$-fixed point, and we denote the coordinates on $U_{\alpha} \cong \C^2$ by $(x_\alpha,y_\alpha)$. Then $T$ acts on the coordinate functions $x_\alpha, y_\alpha$ by characters which we denote by
$$
(w_\alpha)_1 \colon T \to \C^*, \quad (w_\alpha)_2 \colon T \to \C^*.
$$
We write $w_\alpha = ((w_\alpha)_1, (w_\alpha)_2)$, viewed as functions in the coordinates $t_1,t_2$ of $T$. Moreover, the restrictions $L_i|_{U_\alpha}$ correspond to characters, which we denote by
$$
L_i^{(\alpha)} \colon T \to \C^*.
$$
\begin{proposition} \label{prop:main}
For any smooth projective toric surface $S$ and $T$-equivariant algebraic classes $\bfa  = (a_1, \ldots, a_{r-1})\in H^2_T(S,\Z)^{r-1}$ satisfying $a_i^2 = a_i K_S$ for all $i$, we have
\begin{align*}
&A^{\chi(\O_S)} B^{K_S^2} \prod_{i \leq j} C_{ij}^{a_i a_j} = \\
&\quad \quad \quad \quad  \Upsilon_{S,\boldsymbol{a}} q^{-\frac{r\chi(\O_S) }{2} + \frac{r K_S^2}{24} + Q(\boldsymbol{a})} \Bigg( \prod_{\alpha=1}^{e(S)} \sum_{n=0}^{\infty} q^n \widehat{\chi}_{-y}(M_{\PP^2}(r,n)) \Big|_{(w_\alpha,\{L_i^{(\alpha)}y^{-i}\}_i)} \Bigg)\Bigg|_{t_1=t_2=1},
\end{align*}
where the constant $\Upsilon_{S,\boldsymbol{a}}$ is defined in \eqref{def:Upsiloncst}.
\end{proposition}

The evaluation in the $\alpha$th factor of the product means that we substitute $t_i \mapsto (w_\alpha)_i$ for $i=1,2$ and $e_i \mapsto L_i^{(\alpha)} y^{-i}$ for $i=0, \ldots, r-1$. While each of the $e(S)$ generating series on the right-hand side has poles at $t_1=t_2=1$,  these poles cancel after taking the product over all toric charts, and each term of the product can be regarded as a Laurent polynomial in $t_1,t_2$. The specialization to $t_1=t_2=1$ is therefore well-defined.

\subsection{Wall-crossing for framed sheaves}
To deduce new constraints on Vafa-Witten invariants, we combine Proposition \ref{prop:main} with the following two identities for the series \eqref{eq:adjSeries}, the K-theoretic $\SU(r)$ Nekrasov partition function with adjoint matter.
\subsubsection{A blow-up formula}

The first of the two identities is a recent blow-up formula of Kuhn-Leigh-Tanaka \cite{KLT}, which we now recall.

Consider the lattices $A_{r}$ and $A_r^\vee$ defined as follows. Denoting the standard basis of $\Z^r$ by $(e_1, \ldots, e_r)$, the symmetric bilinear form of $A_{r}$ is determined by
$$
\langle e_i, e_j \rangle := \left\{ \begin{array}{cc} 2 & \textrm{if  } i=j \\ -1 & \textrm{if } |i-j|=1 \\ 0 & \textrm{otherwise.} \end{array}\right.
$$
Denote the matrix corresponding to this bilinear form by $M_{A_r}$. Then, in terms of the standard basis of $\Z^r$, the symmetric bilinear form $\langle - , - \rangle^{\vee}$ of $A_{r}^\vee$ is given by the matrix $M_{A_r^\vee} := M_{A_r}^{-1}$. Denoting the upper half plane by $\mathfrak{H} \subset \C$, the associated lattice theta functions are meromorphic functions on $\mathfrak{H} \times \C$. Their Fourier expansions in $q = \exp(2 \pi \sqrt{-1} \tau)$, $y = \exp(2 \pi \sqrt{-1} z)$ are given by
\begin{align*}
\Theta_{A_{r},\ell}(q,y) &:= \sum_{v \in \mathbb{Z}^r} q^{\frac{1}{2} \langle v - \ell \lambda, v - \ell \lambda \rangle} y^{\langle v - \ell \lambda, M_{A_{r}}^{-1} (1, \ldots, 1) \rangle}, \quad  \lambda := \frac{1}{r+1}(r,r-1, \ldots, 1) \\
\Theta_{A_{r}^\vee,\ell}(q,y) &:= \sum_{v \in \mathbb{Z}^r} q^{\frac{1}{2} \langle v, v \rangle^\vee} e^{2 \pi \sqrt{-1} \langle v, \ell \cdot (1,0, \ldots, 0) \rangle^\vee} y^{\langle v, (1, \ldots, 1) \rangle^\vee}. 
\end{align*}
Furthermore, we consider the Dedekind eta function and its normalized version
\begin{align*}
\eta(q) := q^{\frac{1}{24}} \prod_{n=1}^{\infty}(1-q^n), \quad \overline{\eta}(q) := \prod_{n=1}^{\infty}(1-q^n).
\end{align*}

Let $\pi \colon \widehat{\PP}^2 \to \PP^2$ denote the blow-up at the point $[1:0:0]$ and  $C$ denote the exceptional divisor. Note that $[1:0:0]$ does not lie on the line $\ell_\infty \subset \PP^2$. Let $M_{\widehat{\PP}^2}(r,\ell,n)$ be the moduli space parametrizing isomorphism classes of framed sheaves $(E, \Phi)$ on $\widehat{\PP}^2$ where $E$ is a rank $r$ torsion free sheaf on $\widehat{\PP}^2$, locally free in a neighborhood of $\pi^{-1}(\ell_\infty) =: \ell_\infty$, with $c_1(E) = \ell C$, $c_2(E) = n-\frac{r-1}{2r}\ell^2$, and 
$$
\Phi \colon E|_{\ell_\infty} \stackrel{\cong}{\rightarrow} \O_{\ell_\infty}^{\oplus r}.
$$
Then $M_{\widehat{\PP}^2}(r,\ell,n)$ is a smooth quasi-projective variety of dimension 
$2rn.$

The following theorem, \cite[Theorem 1.4]{KLT}, 
is proved using the blow-up algorithm \cite[Figure 1]{NY3} whose key input is the wall-crossing formula \cite[Theorem 1.5]{NY3}.
\begin{theorem}[Kuhn-Leigh-Tanaka] \label{thm:KLT}
For any $\ell \in \Z$, we have
$$
\sum_{n} \widehat{\chi}_{-y}(M_{\widehat{\PP}^2}(r,\ell,n)) q^{n } =  \frac{\Theta_{A_{r-1},\ell}}{\overline{\eta}^r} \sum_{n} \widehat{\chi}_{-y}(M_{\PP^2}(r,n)) q^n.
$$
\end{theorem}

\subsubsection{Stable/co-stable wall-crossing}

The second of the two identities for \eqref{eq:adjSeries} is new. Namely, we establish the following symmetry relation
\begin{align} \label{eqn:framingsymmintro}
\widehat{\chi}_{-y}(M_{\PP^2}(r,n)) \Big|_{(t_1,t_2,e_0^{-1},\ldots, e_{r-1}^{-1})} = \widehat{\chi}_{-y}(M_{\PP^2}(r,n)). \end{align}

The identity \eqref{eqn:framingsymmintro} admits a geometric interpretation. The moduli space $M_{\PP^2}(r,n)$ is a Nakajima quiver variety -- in this case the ADHM moduli space $M(r,n)$ of stable representations of a quiver with relations (Section \ref{sec:quiver}). 
Replacing the stability condition by its opposite, one obtains the moduli space $M^c(r,n)$ of \emph{co-stable} representations of the same quiver with relations. The invariance \eqref{eqn:framingsymmintro} can be rephrased as 
the following assertion that the stable/co-stable wall-crossing is trivial.
\begin{theorem}\label{thm:stcostintro} 
There is an equality \begin{align}\label{stcostintro} \chi(M(r,n),\Omega^k_{M(r,n)} ) =\chi(M^{c}(r,n), \Omega^k_{M^{c}(r,n)})\in \Q(t_1,t_2,e_0,\ldots,e_{r-1})
\end{align} of  $\T$-equivariant Euler characteristics. 
\end{theorem}

The cohomological specialization of \eqref{stcostintro} was previously derived by Ohkawa \cite{O} using Mochizuki style wall-crossing. Our proof of Theorem \ref{thm:stcostintro} uses equivariant mixed Hodge modules and a property of the BPS sheaf $\mathcal{BPS}_{\widehat{A}_0}$ established in \cite{DHSM}. The use of mixed Hodge modules to study Hodge polynomials of Hilbert schemes dates back to \cite{GS}. More recent applications can be found in \cite{Fu}.

 Theorem \ref{thm:stcostintro} in fact holds more generally, for any pair of smooth Nakajima quiver varieties related by variation of GIT stability (see Remark \ref{rmk:generalNak}). Identities of the form \eqref{stcostintro} may be of interest beyond their applications to Vafa-Witten invariants. In particular, applying K-theoretic equivariant localization \cite[Theorem 3.5]{Thom} to both sides of \eqref{stcostintro} 
 yields new non-trivial combinatorial identities. 

As Nakajima varieties are holomorphic symplectic, Theorem \ref{thm:stcostintro} can be regarded as a non-compact and equivariant analog of results of Batyrev \cite{Bat} for $K$-trivial smooth projective varieties and Huybrechts \cite{Huy} for compact hyperk\"{a}hler manifolds.

\subsection{Relations among the universal series}

We now return to the vertical branch and the universal functions $B, C_{ij}$ in Theorem \ref{thm:Laarakker}. We first introduce some notation. For any subset $I \subset [r-1]:=\{1, \ldots, r-1\}$, define\footnote{Note that the definition of $C_{I}$ slightly differs from \cite{GKL}, where $C_{I} = C_0 \prod_{i \leq j \in I} C_{ij}$. Hence $C_0$ and $B$ differ by a factor of $\Theta_{A_{r-1},0}/ \eta^r$.}
\begin{align} \label{def:CI} 
C_{I} := B \prod_{i \leq j \in I} C_{ij}.
\end{align} 
In particular, $C_{\varnothing} = B$. We observe that the universal functions $C_I$ determine the universal functions $B, C_{ij}$ (and vice versa). Furthermore, we define
$$
|I| := \sum_{i \in I} 1, \quad |\!|I|\!| := \sum_{i \in I} i, \quad \epsilon_r := e^{2 \pi \sqrt{-1}/r}.
$$
Our main result is a proof of \cite[Conjecture 3.6]{GKL}, which provides a system of \emph{symmetry} and \emph{blow-up} formulas for the vertical universal functions. 
\begin{theorem} \label{thm:main}
For any $r, \ell$, and $i \leq j$, we have  
\begin{align*}
C_{ij} = C_{r-j,r-i}, \quad \sum_{I \subset [r-1]} \epsilon_r^{\ell |\!|I|\!|} C_I^{-1} = \frac{\Theta_{A_{r-1}^{\vee},\ell}}{\eta^r}.
\end{align*}
\end{theorem}
The former equations are proved by combining Proposition \ref{prop:main} with Theorem \ref{thm:stcostintro}; the latter by combining Proposition \ref{prop:main} with Theorem \ref{thm:KLT}. Some finesse is required when  transferring results from $M_{\PP^2}(r,n)$ to $\Hilb_{\bfbeta}^{\bfn}(S)$ for smooth projective $S$.
In Theorem \ref{thm:main} we do not obtain new equations when replacing $\ell$ by $\ell + kr$ or $-\ell$.  
Thus, without loss of generality, we may take $\ell = 0, \ldots, \lfloor r/2 \rfloor.$ 
The relations $C_{ij} = C_{r-j,r-i}$ are important for verifying $S$-duality in rank $r>2$, see \cite[Section 4.3]{GKL}. 

Since the universal series $A$ is known, this set of equations determines all universal functions for $r=2$
$$
A = \frac{-1}{\phi_{-2,1}(q^2,y^2)^{\frac{1}{2}} \Delta(q^2)^{\frac{1}{2}}}, \quad B = \Bigg( \frac{\Theta_{A_{1},0}}{\eta^2} \Bigg)^{-1}, \quad C_{11} = \frac{\Theta_{A_{1},0}}{\Theta_{A_{1},1}}.
$$
\begin{corollary}
For any smooth polarized surface $(S,H)$ satisfying $H_1(S,\Z) = 0$ and $p_g(S)>0$ and $c_1 \in H^2(S,\Z)$, we have
\begin{align*}
\frac{\sfZ_{S,H,c_1}^{(1^2)}}{(y^{\frac{1}{2}} - y^{-\frac{1}{2}})^{\chi(\O_S)}}  = \Bigg( \frac{-1}{\phi_{-2,1}(q^2,y^2)^{\frac{1}{2}} \Delta(q^2)^{\frac{1}{2}}} \Bigg)^{\chi(\O_S)} \Bigg( \frac{\Theta_{A_{1},0}}{\eta^2} \Bigg)^{-K_S^2}  \sum_{a}  \delta_{c_1,a} \SW(a)  \Bigg( \frac{\Theta_{A_{1},0}}{\Theta_{A_{1},1}} \Bigg)^{a^2}.
\end{align*}
\end{corollary}

Setting $y=1$, this corollary provides a mathematical proof of the \emph{vertical contribution} of Vafa-Witten's original formula, i.e., line 1 of \cite[(5.38)]{VW} (and its generalization by Dijkgraaf-Park-Schroers, i.e., line 1 of \cite[(6.1)]{DPS}). For general $y$, it proves \cite[Remark 1.7]{GK2}.

For general $r$, the relations of Theorem \ref{thm:main} are not enough to determine all universal functions appearing in Theorem \ref{thm:Laarakker}. For $r=3$, there is 1 universal function more than the number of known relations. For $r=4$, there are 2 universal functions more than the number of relations, and for $r=5$, there are 4 universal functions more than the number of relations. Conjectural formulae for the missing universal functions for $r=3,4, 5$ can be found in \cite{GK2, GKL}.
\\

\noindent \textbf{Acknowledgments.} We thank M.~Bershtein, L.~Bertsch, B.~Davison, L.~G\"ottsche, L.~Hennecart, N.~Kuhn, O.~Leigh, D.~Maulik,  H.~Nakajima, R.~Ohkawa, {J.~Sch\"{u}rmann,} Y.~Tanaka, R.P.~Thomas, Y.~Zhao and Z.~Zhou for helpful discussions.  The first- and third-named authors were supported by the EPSRC through grant EP/R013349/1. The first-named author also was supported by the NSF through grant DMS-19027 and the World Premier International Research Center Initiative (WPI), MEXT, Japan. The second-named author is supported by NWO grant VI.Vidi.192.012 and ERC Consolidator Grant FourSurf 101087365. We also thank the organizers of the workshop \emph{New four-dimensional gauge theories} at MSRI Berkeley, where part of this work was carried out.

\section{Nested Hilbert schemes}\label{sec:GT}

In this section, we review and rephrase the required results from \cite{GT1, GT2, Laa1}. After recalling notation from K-theory, we review how to write the vertical contribution $\mathsf{Z}^{(1^r)}_{S,H,c_1}$ in terms of explicit tautological integrals over products of Hilbert schemes on projective surfaces. We then perform a series of reductions. We first use a universality result to recover the values of these integrals for arbitrary projective $S$ from their values for toric $S$. Using equivariant localization, we further reduce to an equivariant analog when $S$ is $\C^2$.

\subsection{K-theory preliminaries}

\subsubsection{} Given a $\C$-scheme $X$ of finite type and action of a torus $T$, we let $K^0_T(X)$ denote the Grothendieck group of $T$-equivariant locally free sheaves and $K_0^T(X)$ denote the Grothendieck group of $T$-equivariant coherent sheaves on $X$. For $X$ a smooth variety, the natural group homomorphism  
$K^0_T(X)\to K_0^T(X)$ is an isomorphism. The group $K^0_T(X)$ carries an algebra structure with respect to which $K_0^T(X)$ carries a module structure. When $T$ is trivial, we denote these groups by $K^0(X)$ and $K_0(X)$, respectively.

\subsubsection{}
For a $T$-equivariant vector bundle $V$ on a $\C$-scheme $N$ of finite type, define 
$$
\Lambda_{-y} V := \sum_{i=0}^{\rk(V)}  [\Lambda^i V] \cdot (-y)^i \in K^0_T(N)[y]. 
$$
Here, $y$ can denote either a formal parameter or a non-trivial $T$-weight. As explained in \cite[2.1.8]{Liu1}, the splitting principle and the identity $$\frac{1}{1-yL}=\frac{1/(1-y)}{1+\frac{(1-L)}{1-y}y }=\sum_{m\geq 0} \frac{(-y)^m}{(1-y)^{m+1}}(1-L)^{\otimes m}$$ for line bundles $L$ together imply that the operator $\Lambda_{-y}$ extends to an operator $$K^0_T(N)\to K^0_T(N)\Big[\frac{1}{1-y}\Big].$$ 
We will also use the following specialization. Suppose that $T$ acts trivially on a $\C$-scheme $N$ of finite type. For $T$-equivariant vector bundles $V$ and $W$ on $N$, write $$V-W=\sum_{w\in T^{\vee}} w\cdot (V_w-W_w)\in K^0_T(N)$$ where $V_w$ and $W_w$ denote the $w$-weight spaces of $V$ and $W$, respectively. Then, given a class $V-W\in K^0_T(N)$ such that $W_{1}=0$, 
define $$\Lambda_{-1}(V-W):=\bigotimes_{w} \Lambda_{-w}(V_w-W_w) \in K^0_T(N)\Big[\frac{1}{1-w}\Big]_{w\neq 1}$$ and $$ \frac{1}{\Lambda_{-1}(W-V)}:=\Lambda_{-1}(V-W).$$

\subsection{Reduction to $\Hilb^{\bfn}(S)$}

Let $(S,H)$ be a smooth polarized surface satisfying $H_1(S,\Z) = 0$, $p_g(S)>0$, and let $r>1$.
As discussed in the introduction, for any effective algebraic $\bfbeta \in H^2(S,\Z)^{r-1}, \bfn \in \Z_{\geq 0}^{r}$, a perfect obstruction theory on the nested Hilbert scheme of points and curves $\Hilb_{\bfbeta}^{\bfn}(S)$ is constructed in \cite{GT1, GT2}. Let $\I_i$ be the pull-back along projection to $S \times \Hilb^{\bfn}(S) \times |\bfbeta|$ of the universal ideal sheaf on $S \times \Hilb^{n_i}(S)$. Let $\pt = \mathrm{Spec}(\C)$ and consider the projection and inclusion
$$
\pi \colon \Hilb^{\bfn}(S)\times S \times |\bfbeta| \to \Hilb^{\bfn}(S) \times |\bfbeta|, \quad \iota \colon \Hilb^{\bfn}_{\bfbeta}(S) \hookrightarrow \Hilb^{\bfn}(S) \times |\bfbeta|.
$$ 
We denote by $R\hom_\pi(-,-)$ the derived composition $R\pi_* R\hom(-,-)$.
\begin{theorem}[Gholampour-Thomas] \label{thm:GT}
The push-forward 
$$
\iota_* [\Hilb_{\bfbeta}^{\bfn}(S)]^{\vir} \in H_{2n_0+2n_{r-1}}(\Hilb^{\bfn}(S) \times |\bfbeta|,\Z)
$$ 
equals\footnote{The Euler classes in the product are well-defined by \cite[Theorem 8.3]{GT1}, i.e., by generalized Carlsson-Okounkov vanishing. In particular, for $\mathbb{F}_i := R \pi_* \O(\beta_i)  - R \hom_{\pi}(\I_{i-1},\I_i(\beta_i))$, we have $c_{>(n_{i-1}+n_i)}(\mathbb{F}_i) = 0$ and $e(\FF_i) := c_{n_{i-1}+n_i}(\FF_i)$, where $\rk(\FF_i) = n_{i-1}+n_i$.}
\begin{align*}
\prod_{i=1}^{r-1} \SW(\beta_i) \cdot e\Big(R\pi_*\O(\beta_i)- R \hom_{\pi}(\I_{i-1},\I_i(\beta_i))\Big) \cap [\Hilb^{\bfn}(S) \times \pt \times \cdots \times \pt],
\end{align*}
where $\pt \times \cdots \times \pt \in |\bfbeta|$.
\end{theorem}

\begin{remark}
We expect that Theorem \ref{thm:GT} can be upgraded to virtual structure sheaves. Namely, we expect that  $\iota_* \O_{\Hilb^{\bfn}_{\bfbeta}(S)}^{\vir} \in K_0( \Hilb^{\bfn}(S) \times |\bfbeta|)$ equals
\[
\mathcal{G}:=\prod_{i=1}^{r-1} \SW(\beta_i) \cdot \Lambda_{-1} \Big(R\pi_*\O(\beta_i)- R \hom_{\pi}(\I_{i-1},\I_i(\beta_i))\Big)^{\vee} \cdot \O_{\Hilb^{\bfn}(S) \times \pt \times \cdots \times \pt}.
\]

 By \cite[(4.27)]{GT2}, the K-class $R\pi_*\O(\beta_i)- R \hom_{\pi}(\I_{i-1},\I_i(\beta_i))$ can be represented by a vector bundle after pull-back to an affine bundle. The Thom isomorphism theorem then implies that the exterior products appearing in the expression for $\mathcal{G}$ are well-defined.

For our results, however, it is sufficient to use the following ``numerical'' K-theoretic version of Theorem \ref{thm:GT}: for any $\mathcal{F}\in K^0( \Hilb^{\bfn}(S) \times |\bfbeta|)$ one has \begin{align*}
\chi&\Big( \Hilb^{\bfn}(S) \times |\bfbeta|, \mathcal{F}\otimes \iota_* \O_{\Hilb^{\bfn}_{\bfbeta}(S)}^{\vir}\Big)=\chi( \Hilb^{\bfn}(S) \times |\bfbeta|, \mathcal{F}\otimes \mathcal{G}).
\end{align*}  This equality follows from the virtual Hirzebruch-Riemann-Roch formula (\cite[Theorem 4.5.1]{CFK} and \cite[Corollary 3.4]{FG}), as well as  Theorem \ref{thm:GT} and \cite[(4.27)]{GT2}.
\end{remark}

Fix $c_1 \in H^2(S,\Z)$. Define line bundles $L_0, \ldots, L_{r-1}$ on $S$ determined (up to isomorphism) by the following equations
\begin{equation} \label{eqn:beta}
K_S - \beta_i = c_1(L_{i-1} \otimes L_i^*), \quad \sum_{i=0}^{r-1} c_1(L_i) = c_1.
\end{equation}

Consider the line bundle $\O_{|\beta_i|}(1)$ on the complete linear system $|\beta_i|$. We form the following line bundles on $S \times \Hilb^{\bfn}(S) \times |\bfbeta|$
\begin{align*}
\L_0 &= L_0, \\ 
\L_1 &= L_1 \boxtimes \O_{|\beta_1|}(1), \\
&\cdots \\
\L_{r-1} &= L_{r-1} \boxtimes \O_{|\beta_{1}|}(1) \boxtimes \cdots \boxtimes \O_{|\beta_{r-1}|}(1).
\end{align*}
Then $\O_S(\beta_i) \boxtimes \O_{|\beta_i|}(1)$ have tautological sections inducing maps $$\L_{i-1} \to \L_i \boxtimes K_S.$$ We obtain a $\C^*$-equivariant Higgs pair on $S \times \Hilb^{\bfn}(S) \times |\bfbeta|$
$$
\Phi_{\L} \colon \mathbb{E}_{\L} \to \mathbb{E}_{\L} \boxtimes K_S \otimes y, \quad \mathbb{E}_{\L} := \bigoplus_{i=0}^{r-1} \L_i \otimes y^{-i}.
$$
Next, we define
\begin{equation} \label{defE}
\mathbb{E} := \bigoplus_{i=0}^{r-1} \I_i \otimes \L_i \otimes y^{-i}.
\end{equation}
Over the incidence locus $S \times \Hilb_{\bfbeta}^{\bfn}(S) \subset S \times  \Hilb^{\bfn}(S) \times |\bfbeta|$, the tautological Higgs field $\Phi_{\L}$ factors through $\mathbb{E}$. We obtain a Higgs pair $(\EE, \Phi)$ on $S \times  \Hilb^{\bfn}_{\bfbeta}(S)$.  This construction holds for any $\bfbeta, \bfn$ and does not involve stability. 

Suppose that $H,r,c_1,c_2$ are given such that there are no rank $r$ strictly $H$-semistable Higgs pairs $(E,\phi)$ on $S$ with $c_1(E) = c_1$ and $c_2(E) = c_2$. We abbreviate by $N:=N_S^H(r,c_1,c_2)$ the moduli space of rank $r$ $H$-stable Higgs pairs $(E,\phi)$ on $S$ with $c_1(E) = c_1$, and $c_2(E) = c_2$.\footnote{As $H_1(S,\Z) = 0$, it is equivalent to fix determinant or first Chern class.} Then $N_{(1^r)}^{\C^*}$ is isomorphic to the union of incidence loci $\Hilb_{\bfbeta}^{\bfn}(S)$ with $\bfbeta  \in H^2(S,\Z)^{r-1}$ and $\bfn \in \Z_{\geq 0}^{r}$ satisfying the following properties: (1) all elements of $\Hilb_{\bfbeta}^{\bfn}(S)$ are $H$-stable, and (2) $\bfn, \bfbeta$ satisfy
\begin{align}
\begin{split} \label{Cherneq}
&c_1 \equiv \sum_{i=1}^{r-1} i(K_S - \beta_i) \mod rH^2(S,\Z), \\
&c_2 = |\bfn| + \frac{r-1}{2r} c_1^2 + Q(K_S - \beta_1, \ldots, K_S - \beta_{r-1}),
\end{split}
\end{align}
where $|\bfn| := \sum_{i} n_i$ and $Q$ is as  defined in \eqref{def:Q}. For proofs of these statements, we refer to \cite{GT2, Laa1}. Then the contribution of the component $\Hilb_{\bfbeta}^{\bfn}(S) \subset N_{(1^r)}^{\C^*}$ to the K-theoretic Vafa-Witten invariant \eqref{def2} is shown in \cite{GT2} to be 
\begin{equation} \label{eqn:contrdegloc}
\chi\Bigg(\Hilb_{\bfbeta}^{\bfn}(S),\O_{\Hilb_{\bfbeta}^{\bfn}(S)}^{\vir} \otimes \frac{(K^{\vir}_{N})^{\frac{1}{2}}}{\Lambda_{-1} (\nu_N^{\vir})^\vee} \Bigg|_{\Hilb_{\bfbeta}^{\bfn}(S)} \Bigg).
\end{equation}
By the virtual Riemann-Roch theorem, 
this holomorphic Euler characteristic can also be expressed as an integral over the class $[\Hilb_{\bfbeta}^{\bfn}(S)]^{\vir}$ of Theorem \ref{thm:GT}.
In $K^{\C^*}_0(\Hilb_{\bfbeta}^{\bfn}(S))$ we have the following equality from \cite[Section 4]{Laa1}
$$
T_N^{\vir}|_{\Hilb_{\bfbeta}^{\bfn}(S)} = \Big( R \hom_\pi(\mathbb{E}, \mathbb{E} \boxtimes K_S \otimes y)_0 - R \hom_\pi(\mathbb{E}, \mathbb{E})_0 \Big)\Big|_{\Hilb_{\bfbeta}^{\bfn}(S)}.
$$

Here we only considered the incidence schemes $\Hilb_{\bfbeta}^{\bfn}(S)$ for which the corresponding Higgs pairs are $H$-stable. However, crucially, it is shown in \cite[Proposition 3.5]{Laa1} that if $\Hilb_{\bfbeta}^{\bfn}(S)$ contains an $H$-unstable element, then
$$
\iota_* [\Hilb_{\bfbeta}^{\bfn}(S)]^{\vir} = 0.
$$
Therefore, we can sum the contributions \eqref{eqn:contrdegloc} over \emph{all} $\bfbeta, \bfn$ satisfying \eqref{Cherneq}. 

Suppose $r,c_1$ are fixed such that there are no rank $r$ strictly $H$-semistable Higgs pairs $(E,\phi)$ on $S$ with $c_1(E) = c_1$. By Theorem \ref{thm:GT}, one can now express $\sfZ_{S,H,c_1}^{(1^r)}$ in terms of holomorphic Euler characteristics over $\Hilb^{\bfn}(S) \times \pt \times \cdots \times \pt \subset  \Hilb_{\bfbeta}^{\bfn}(S)$.

\subsection{A class on the Hilbert scheme}

We consider the class 
\begin{equation} \label{def:T[n]}
T^{[\bfn]} := R \hom_{\pi}( \E^{[\bfn]},\E^{[\bfn]} \boxtimes K_S \otimes y)_0 - R \hom_{\pi}( \E^{[\bfn]},\E^{[\bfn]})_0\in K^0_{\C^*}(\Hilb^{\bfn}(S)), 
\end{equation}
where
\begin{equation} \label{def:E[n]}
\E^{[\bfn]} :=  \bigoplus_{i=0}^{r-1} \I_i \boxtimes L_i \otimes y^{-i}.
\end{equation}
We also write $\Omega^{[\bfn]} := (T^{[\bfn]})^\vee$. Next, we define
\begin{equation} \label{def:ai}
a_i := K_S - \beta_i, \quad i=1, \ldots, r-1.
\end{equation}
Using the normalized complexes (where we suppress pull-back from a point)
\begin{equation} \label{def:T0[n]}
T^{[\bfn]}_0 := T^{[\bfn]} - T^{[\boldsymbol{0}]}, \quad \Omega^{[\bfn]}_0 = (T^{[\bfn]}_0)^\vee,
\end{equation}
we set
\begin{align} \label{eqn:defGS}
\mathsf{G}_{S,\boldsymbol{a}} :=  \sum_{\bfn \in \Z_{\geq 0}^r} q^{|\bfn|}  \chi(\Hilb^{\boldsymbol{n}}(S),  \Upsilon_{S,\boldsymbol{a},\boldsymbol{n}})\in \Q(y^{\frac{1}{2}})[[q]]
\end{align}
for 
\begin{align}\label{eqn:defUps}
\Upsilon_{S,\boldsymbol{a},\boldsymbol{n}}: =
\frac{\prod_{i} \Lambda_{-1} \Big(R\pi_*\O(\beta_i)- R \hom_{\pi}(\I_{i-1},\I_i(\beta_i)) \Big)^{\vee}}{\Lambda_{-1} (\Omega_0^{[\bfn]})^m}\otimes \det(\Omega_0^{[\bfn]})^{\frac{1}{2}}
\end{align}
where $(-)^m$ denotes the $\C^*$-moving part, and for convenience we take  $$\det(\Omega^{[\bfn]}_0)^{\frac{1}{2}}= \det\Big(y^{-\frac{1}{2}}R\hom_{\pi}(\E^{[\bfn]},\E^{[\bfn]})_0\Big)\in \Pic_{\C_{y^{1/2}}^*}(\Hilb^{\bfn}(S))$$ to be an explicit choice of square root of $\det(\Omega^{[\bfn]}_0)$ which may require passing to a double cover $\C^*_{y^{1/2}}$ of $\C^*$. By Grothendieck-Riemann-Roch the value of $\mathsf{G}_{S,\bfa}$ does not depend on the choice of square root. We recall that $\beta_i = K_S - a_i$. Note that $$\prod_{i} \Lambda_{-1} \Big( R\pi_*\O(\beta_i) - R \hom_{\pi}(\I_{i-1},\I_i(\beta_i))  \Big)^{\vee}= \Lambda_{-1}\Big(T_{\Hilb^{\bfn}(S)} - (T^{[\bf n]}_0)^{f}\Big)^\vee,$$ where $(-)^f$ denotes the $\C^*$-fixed part. 
We thus obtain the following well-defined expression
\begin{align}\label{def:altUps}
\Upsilon_{S,\boldsymbol{a}, \bfn} = \Lambda_{-1}(-\Omega_0^{[\bfn]}+\Omega_{\Hilb^{\bfn}(S)}) \otimes \det(\Omega_0^{[\bfn]})^{\frac{1}{2}} \in K^0_{\C^*}(\Hilb^{\bfn}(S))\Big[y^{\frac{1}{2}},\frac{1}{1-y^{k}}\Big]_{k\in \Z_{\neq0}}.
\end{align}

By Hirzebruch-Riemann-Roch, the K-theoretic expression $\chi(\Hilb^{\bfn} (S), \Upsilon_{S,\boldsymbol{a},\boldsymbol{n}})$ defined in \eqref{eqn:defGS} and \eqref{eqn:defUps} is equal to the cohomological expression $Q_{\bfn}(S,{\boldsymbol{\beta}},y)$ defined in \cite[Section 5]{Laa1}. 

To handle the constant factor arising from the pull-back of the complex $T^{[\boldsymbol 0]}$, we set \begin{align} \label{def:Upsiloncstorigin}
\Upsilon_{S,\boldsymbol{a}} :=  \chi\Big(\pt, \frac{\det(\Omega^{[\boldsymbol{0}]})^{\frac{1}{2}}}{\Lambda_{-1}(\Omega^{[\boldsymbol{0}]})^{m}}\Big)(y^{\frac{1}{2}} - y^{-\frac{1}{2}})^{-\chi(\O_S)} \prod_i (-1)^{(r-1) i a_i^2 } \in \Q(y^{\frac{1}{2}}).
\end{align}
The term $\Upsilon_{S,\boldsymbol{a}}$ is equal to the quantity $F(S,\boldsymbol{\beta},y)$ defined in \cite[Section 5]{Laa1}. 

The main result of \cite[Sections 4--6]{Laa1} is that
\begin{align}
\begin{split} \label{eqn:Laaeqn}
\frac{\sfZ_{S,H,c_1}^{(1^r)}}{(y^{\frac{1}{2}} - y^{-\frac{1}{2}})^{\chi(\O_S)}} = q^{-\frac{r}{2} \chi(\O_S) + \frac{r}{24}K_S^2} \sum_{\boldsymbol{a}} \delta_{c_1,\sum_i i a_i}  q^{Q(\boldsymbol{a})} \Upsilon_{S,\boldsymbol{a}} \mathsf{G}_{S,\boldsymbol{a}}  \prod_i \SW(a_i), 
\end{split}
\end{align}
where $\delta_{a,b}$ is defined in \eqref{def:delta}. The relation $\SW(a_i) = (-1)^{\chi(\O_S)} \SW(\beta_i)$ is used in the argument for the above result.

\subsubsection{} 
The definitions of $\Upsilon_{S,\boldsymbol{a},\boldsymbol{n}}$ and $\mathsf{G}_{S,\boldsymbol{a}}$ may be extended to allow $S$ to be any (possibly disconnected) smooth projective surface and $\boldsymbol{a} \in H^2(S,\Z)^{r-1}$ to be any collection of algebraic classes. In particular, it  need not be the case that $a_i^2 = a_i K_S$ for all $i=1, \ldots r-1$.

In this setting, given such $S$ and $\bfa$, we choose $L_0,\ldots,L_{r-1}\in \Pic(S)$ such that $a_i=c_1(L_{i-1}\otimes L_i^*)$ for $1\leq i <r$.  We set $\beta_i=K_{S}-a_i$ and define the generating series $\mathsf{G}_{S,\boldsymbol{a}}$ using \eqref{def:T[n]} \eqref{def:E[n]} \eqref{def:T0[n]}, \eqref{eqn:defGS} and \eqref{eqn:defUps}. Note that the series $\mathsf{G}_{S,\bfa}$ does not depend on the above choice of $L_i$ and, in contrast to \eqref{eqn:beta}, we do not fix a $c_1\in H^2(S,\Z)$ when defining $\mathsf{G}_{S,\bfa}$.
The series $\mathsf{G}_{S,\bfa}$ has constant coefficient $1$.

\subsection{Universality} 

For any $0 \leq m \leq n$, we define the quantum number and quantum binomial coefficient
\begin{align*}
[n]_y &:= y^{-(n-1)/2} + y^{-(n-3)/2} + \cdots + y^{(n-3)/2} + y^{(n-1)/2}, \\
\binom{n}{m}_y &:= \frac{[n]_y \cdots [n-m+1]_y}{[1]_y \cdots [m]_y}.
\end{align*}
Suppose that $a_i^2 = a_i K_S$ for all $i$. This case is the only one of interest, because if $\SW(a_i) \neq 0$ then $a_i^2 = a_i K_S$ \cite{Moc}.  In this case, the quantity $\Upsilon_{S,\boldsymbol{a}}$ was determined in \cite[Proposition 6.3]{Laa1}
\begin{equation} \label{def:Upsiloncst}
 \Upsilon_{S,\boldsymbol{a}}= \Bigg(\frac{(-1)^{r-1}}{[r]_y (y^{\frac{1}{2}} - y^{-\frac{1}{2}})} \Bigg)^{\chi(\O_S)} \prod_{i=1}^{r-1} \binom{r}{i}^{-a_i^2}_y \prod_{1 \leq i < j \leq r-1} \Bigg( \frac{[j]_y[r-i]_y}{[j-i]_y[r]_y} \Bigg)^{a_i a_j}.
\end{equation}

Moreover, for a disconnected surface $S = S' \sqcup S''$ we have
$$
\Hilb^{n}(S) \cong \bigsqcup_{n = n' + n''} \Hilb^{n'}(S') \times \Hilb^{n''}(S'').
$$
By a standard argument, cf.~\cite[Section 7]{Laa1}, and by \eqref{def:Upsiloncst} we find the following. 
\begin{lemma}\label{lem:Mult}
If $S = S' \sqcup S''$ and $\boldsymbol{a} = \boldsymbol{a}' \oplus \boldsymbol{a}''$, then
\begin{align*}
\mathsf{G}_{S,\boldsymbol{a}} &= \mathsf{G}_{S',\boldsymbol{a}'} \mathsf{G}_{S'',\boldsymbol{a}''},
\\
\Upsilon_{S,\boldsymbol{a}}&=\Upsilon_{S',\boldsymbol{a}'} \Upsilon_{S'',\boldsymbol{a}''}.
\end{align*}
\end{lemma}

In \cite[Section 7]{Laa1}, this multiplicative property is combined with the argument of \cite[Theorem 5.1]{GNY}, which generalizes the Ellingsrud-G\"ottsche-Lehn universality theorem \cite[Theorem 4.2]{EGL} to products of Hilbert schemes.

Thus, there exist universal power series $\overline{A}$, $\overline{B}$, $\{\overline{E}_i\}_{i=1}^{r-1}$, $\{\overline{E}_{ij}\}_{1 \leq i \leq j \leq r-1}$, depending only on $r$ and with constant coefficient 1, such that
\begin{equation} \label{eqn:univdef}
 \mathsf{G}_{S,\boldsymbol{a}} = \overline{A}^{\chi(\O_S)} \overline{B}^{K_S^2} \prod_i \overline{E}_i^{a_i K_S} \prod_{i \leq j} \overline{E}_{ij}^{a_i a_j}.
\end{equation}
Recall that the generating series $\mathsf{G}_{S,\boldsymbol{a}}$ only contributes to the vertical generating series \eqref{eqn:Laaeqn} when $a_i K_S = a_i^2$. So  two universal series entering with exponents $a_iK_S$ and $a_i^2$ can be combined. Hence we define
\begin{align} \label{eqn:defbarCij}
\overline{C}_{ii} := \overline{E}_i \overline{E}_{ii}, \quad \overline{C}_{ij} := \overline{E}_{ij},
\end{align}
where in the second definition $i<j$. We furthermore incorporate the constant terms of \eqref{def:Upsiloncst} and define
\begin{align}
\begin{split} \label{eqn:defCij}
A &:= \frac{(-1)^{r-1}}{[r]_y (y^{\frac{1}{2}} - y^{-\frac{1}{2}})}   q^{-\frac{r}{2}} \overline{A}, \\
B &:=  q^{\frac{r}{24}} \overline{B}, \\
C_{ii} &:= \binom{r}{i}_{y}^{-1} q^{-\frac{i(r-i)}{2r}} \overline{C}_{ii}, \\
C_{ij} &:=  \frac{[j]_{y}[r-i]_{y}}{[j-i]_{y}[r]_{y}} q^{-\frac{i(r-j)}{r}} \overline{C}_{ij}, 
\end{split}
\end{align}
and Theorem \ref{thm:Laarakker} follows. Recall that an explicit formula \eqref{eqn:A} for the universal series $A$ was computed in \cite[Section 5.4]{Tho} and \cite[Theorem C]{Laa1}. Then for all $S,a_i$ satisfying $a_i^2 = a_i K_S$, we have
\begin{equation} \label{eqn:final}
q^{-\frac{r\chi(\O_S)}{2} + \frac{r K_S^2}{24} + Q(\boldsymbol{a})} \Upsilon_{S,\boldsymbol{a}} \mathsf{G}_{S,\boldsymbol{a}} = A^{\chi(\O_S)} B^{K_S^2} \prod_{i \leq j} C_{ij}^{a_i a_j}.
\end{equation}

\subsection{Reduction to $\Hilb^{\bfn}(S)$ for toric $S$} \label{sec:HilbToric}

Equation \eqref{eqn:univdef} determines the universal series $\overline{A}$, $\overline{B}$, $\overline{E}_i$, $\overline{E}_{ij}$, and thus $A$, $B$, $C_{ij}$, by evaluation on any collection of pairs $(S,\boldsymbol{a})$ for which the vectors 
$$
(\chi(\O_S), K_S^2, \{a_i K_S\}_i, \{a_ia_j\}_{i \leq j})\in \Q^{2+r-1+\binom{r}{2}}
$$
form a $\Q$-spanning set. 

In particular, one can take a collection for which all $S$ are toric surfaces equipped with $T$-action and $\boldsymbol{a} \in H^2(S,\Z)^{r-1}$ are $T$-equivariant divisors. The $T$-action lifts to the Hilbert scheme $\Hilb^n(S)$ and the products $\Hilb^{\boldsymbol{n}}(S)$.

Recall the notation for smooth projective toric surfaces from Section \ref{sec:intro2}. In particular $\{U_{\alpha}\}_{\alpha=1}^{e(S)}$ denotes  a collection of toric charts, each centered at one of the $e(S)$ points of the fixed locus $S^T$. The torus $T$ acts on $U_{\alpha}$ by scaling the coordinate axes with $T$-weights $(w_{\alpha})_{1}$ and $(w_{\alpha})_2$. 

\subsubsection{} The definitions \eqref{def:T[n]} \eqref{def:E[n]} \eqref{def:T0[n]}, \eqref{eqn:defGS} and \eqref{eqn:defUps} of the localized K-theory classes $\Upsilon_{S,\bfa,n}$ and the series $\mathsf{G}_{S,\boldsymbol{a}}$ can be further  extended verbatim to the setting where $S$ is a finite union of toric charts $U_{\alpha}$, 
the $a_i\in H^2_T(S,\Z)$ are $T$-equivariant classes, and K-theory classes are regarded $T$-equivariantly. In particular, $S$ need not be projective.

In this setting, $$\mathsf{G}_{S,\boldsymbol{a}}\in \Q(y^{\frac{1}{2}},t_1,t_2)[[q]],$$
where the Euler characteristics $\chi(\Hilb^{\bfn}(S),\Upsilon_{S,{\bfa},{\bfn}})$ are \emph{defined} by K-theoretic equivariant localization; that is \begin{align}\label{eq:Gloc} \chi(\Hilb^{\bfn}(S),\Upsilon_{S,{\bfa},{\bfn}}):=&\sum_{p\in \Hilb^{\bfn}(S)^T} \chi\Big(p, \frac{\Upsilon_{S,\bfa,\bfn}|_{p}}{\Lambda_{-1}(\Omega_{\Hilb^{\bfn}(S)}|_p)}\Big) \\ \nonumber =&\sum_{p\in \Hilb^{\bfn}(S)^T} \chi\Big(p, \frac{(\det(\Omega_0^{[\bfn]}|_p))^{\frac{1}{2}}}{\Lambda_{-1}(\Omega_0^{[\bfn]}|_p)}\Big)\in \Q(y^{\frac{1}{2}},t_1,t_2).
\end{align}
This expression is well-defined by \cite[Lemma 6]{CO}, which implies that for any $p\in \Hilb^{\bfn}(S)^T$ and any $i$, the rank $n_{i-1}+n_{i}$ character $$\chi \Big(p,\big(R\pi_*\O(\beta_i) - R \hom_{\pi}(\I_{i-1},\I_i(\beta_i))  \big)\big|_{p} \Big)$$ is a sum of $T$-weights with positive coefficients. 
For projective toric $S$, K-theoretic equivariant localization (\cite[Theorem 3.5]{Thom}) implies that the specialization at $t_1=t_2=1$ of the the right-hand side of \eqref{eq:Gloc} agrees with $\chi(\Hilb^n(S),\Upsilon_{S,\bfa,\bfn})$, as defined in \eqref{eqn:defUps}.

Lemma \ref{lem:Mult} also holds in the setting where each of $S'$ and $S''$ is either a toric chart or a blow-up of a toric chart at its torus fixed point.

\subsection{Reduction to $\Hilb^{\bfn}(\C^2)$} \label{sec:HilbC^2}

If $S$ is a toric smooth projective surface   
with toric charts $\{U_{\alpha}\}_{\alpha=1}^{e(S)}$, then there is a natural bijection between the $T$-fixed points of $\Hilb^{\bfn}(S)$ and $\Hilb^{\bfn}(\sqcup_{\alpha} U_{\alpha})$, under which the corresponding summands of the right-hand sides of \eqref{eq:Gloc} for $\Hilb^{\bfn}(S)$ and $\Hilb^{\bfn}(\sqcup_{\alpha} U_{\alpha})$ coincide.  We conclude that 
$$\mathsf{G}_{S,\bfa}=\big(\mathsf{G}_{\sqcup_{\alpha} U_{\alpha},\ \sum_{\alpha} \bfa|_{U_{\alpha}}}\big)\Big|_{t_1=t_2=1}.$$

Applying Lemma \ref{lem:Mult}, we obtain the following lemma. 

\begin{lemma} \label{lem:prod} 
If $S$ is a toric smooth projective surface and $\bfa\in H^2(S,\Z)^{r-1}$ is a collection of $T$-equivariant divisors, then there is a factorization
$$\mathsf{G}_{S,\bfa}=\Bigg( \prod_{\alpha=1}^{e(S)} \mathsf{G}_{U_{\alpha},\ \bfa|_{U_{\alpha}} } \Bigg)\Bigg|_{t_1=t_2=1}.$$
\end{lemma}

In Section \ref{sec:framedtoVW}, we explain how to write each factor $\mathsf{G}_{U_{\alpha},\ \bfa|_{U_{\alpha}}}$ in terms of $M_{\PP^2}(r,n)$.

Similar statements hold in the non-compact setting. The following will be relevant. Let $\widehat{\C}^2$ denote $\C^2$ blown up at the origin, with induced $T$-action and cover $\widehat{\C}^2$ by toric charts $U_1$ and $U_2$ centered at the $T$-fixed points, so that {on $U_1$ the coordinate directions are scaled by $t_1$ and $t_1^{-1}t_2$, and on $U_2$ by $t_1t_2^{-1}$ and $t_2$ (see Figure \ref{fig}).}
 
\begin{figure}

\begin{tikzpicture}[
  line width=1.2pt,
  >=Stealth
]

\def\L{2}

\coordinate (L) at (-\L,0);
\coordinate (R) at ( \L,0);

\draw[
  postaction={decorate,
    decoration={markings,
      mark=at position 0.5 with {\arrow{Stealth}}
    }
  }
]
(L) -- (0,0)
node[midway,below=4pt] {$t_1^{-1}t_2$};

\draw[
  postaction={decorate,
    decoration={markings,
      mark=at position 0.5 with {\arrow{Stealth}}
    }
  }
]
(R) -- (0,0)
node[midway,below=4pt] {$t_1 t_2^{-1}$};

\draw[
  postaction={decorate,
    decoration={markings,
      mark=at position 0.5 with {\arrow{Stealth}}
    }
  }
]
(L) -- ++(120:\L)
node[midway,left=4pt] {$t_1$};

\draw[
  postaction={decorate,
    decoration={markings,
      mark=at position 0.5 with {\arrow{Stealth}}
    }
  }
]
(R) -- ++(60:\L)
node[midway,right=4pt] {$t_2$};

\node[below left=12pt] at (L) {$U_1$};
\node[below right=12pt] at (R) {$U_2$};

\end{tikzpicture}
\caption{The two toric charts of $\widehat{\C}^2$.}
\label{fig}
\end{figure}
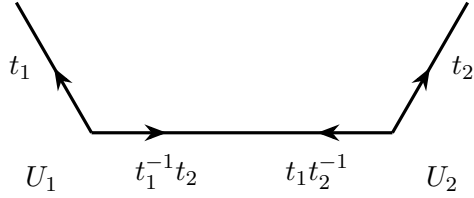

Then the argument of Lemma \ref{lem:prod} also implies that 
\begin{align}\label{eq:GBlfact} \mathsf{G}_{\widehat{\C}^2,\bfa}=\mathsf{G}_{U_1, \bfa|_{U_1}}\cdot \mathsf{G}_{U_2,\bfa|_{U_2}}.
\end{align}

\section{Moduli of framed sheaves} \label{sec:framed}

In this section, we recall the requisite background on moduli spaces of framed sheaves on $\PP^2$ and on the blow up $\widehat{\PP}^2$ of $\PP^2$ at a point. We then state the identities for $\chi_{-y}$-genera that will be applied to Vafa-Witten invariants in Section \ref{sec:framedtoVW}.

\subsection{Framed sheaves on $\PP^2$}\label{sec:MP2}

We consider the moduli spaces $M_{\PP^2}(r,n)$ of framed sheaves from the introduction. We first recall some facts that may be found in \cite{NY1}, for example. The moduli space $M_{\PP^2}(r,n)$ can be equipped with a torus action $\mathbb{T} = T_1 \times T_2 \times T_3 = (\C^*)^{2+r+1}$, where $T_3$ acts trivially. Explicitly, the action of $T_1$ on $\PP^2$ is given by
$$
(t_1,t_2) \cdot [x_0:x_1:x_2] \mapsto [x_0:t_1x_1:t_2x_2]
$$ 
which induces an automorphism $F_{t_1,t_2} \colon \PP^2 \to \PP^2$ for each $(t_1,t_2) \in T_1$. The torus $T_2$ acts on the framing. Then, on closed points, the action of $T_1 \times T_2$ on $M_{\PP^2}(r,n)$ is given by
$$
(t_1,t_2,e_0, \ldots, e_{r-1}) \cdot (E, \Phi) = ((F^{-1}_{t_1,t_2})^* E, \Phi_{(t_1,t_2,e_0, \ldots, e_{r-1})}),
$$
where $\Phi_{(t_1,t_2,e_0, \ldots, e_{r-1})}$ denotes the composition
$$
(F^{-1}_{t_1,t_2})^* E|_{\ell_{\infty}} \stackrel{(F^{-1}_{t_1,t_2})^* \Phi}{\longrightarrow} (F^{-1}_{t_1,t_2})^* \O_{\ell_{\infty}}^{\oplus r} \stackrel{\cong}{\longrightarrow} \O_{\ell_\infty}^{\oplus r} \stackrel{(e_0, \ldots, e_{r-1})}{\longrightarrow} \O_{\ell_\infty}^{\oplus r},
$$
where the middle arrow is induced by the natural $T_1$-equivariant structure on $\O_{\ell_{\infty}}$ and the second arrow is matrix multiplication by $\mathrm{diag}(e_0,\ldots,e_{r-1})$. 

There is a $T_1$-equivariant isomorphism $M_{\PP^2}(1,n)\cong \Hilb^n(\C^2)$. Namely,  elements of $M_{\PP^2}(1,n)$ are of the form $(I_Z,\iota_{Z}|_{\ell_{\infty}}),$ where $I_Z$ is an ideal sheaf of a 0-dimensional subscheme $Z$ supported on $\C^2\subset \PP^2$ and $\iota_{Z}\colon I_{Z}\subset \O_{\PP^2}$ is the inclusion.

\subsection{Framed sheaves on $\widehat{\PP}^2$}

We next consider the moduli spaces of framed sheaves $M_{\widehat{\PP}^2}(r,\ell,n)$ from the introduction and recall facts from \cite{NY1}. Consider the blow-up
$$
\widehat{\PP}^2 = \{ ([z_0:z_1:z_2], [z:w]) \in \PP^2 \times \PP^1 \, : \, z_1 w = z_2 z\}.
$$
Then the exceptional divisor $C:= \{[1:0:0]\} \times \PP^1$ is disjoint from the line $\ell_\infty = Z(z_0)$ at infinity. We write the action of $T_1$ on $\widehat{\PP}^2$ as
$$
(t_1,t_2) \cdot ([z_0:z_1:z_2], [z:w]) = ([z_0:t_1z_1:t_2z_2], [t_1z:t_2w]).
$$
We obtain an induced action of $\mathbb{T} = T_1 \times T_2 \times T_3 = (\C^*)^{2+r+1}$ on $M_{\widehat{\PP}^2}(r,\ell,n)$. We also use the notation $T = T_1$ and $T_3 = \mathbb{C}^*_y$ for the first and third factors.
\newline

For any $\ell\in \Z$,  one has a {$T$-equivariant} isomorphism $M_{\widehat{\PP}^2}(1,\ell,n)\cong \Hilb^n(\widehat{\C}^2)$. Namely, elements of $M_{\widehat{\PP}^2}(1,\ell,n)$ are of the form $(I_Z(\ell C),\iota_{Z}|_{\ell_{\infty}}),$ where $I_Z$ is an ideal sheaf of a 0-dimensional subscheme $Z$ supported on $\widehat{\C}^2\subset \widehat{\PP}^2$ and $\iota_{Z}\colon I_{Z}\subset \O_{\widehat{\PP}^2}$ is the inclusion.

\subsection{The $\chi_{-y}$-genus}

Recall that for a smooth projective variety $X$, the classical (symmetrized) Hirzebruch $\chi_{-y}$ genus is defined by 
$$
\widehat{\chi}_{-y}(X) = y^{-\frac{\dim(X)}{2}} \chi(X,\Lambda_{-y} \Omega_X).
$$

Similarly, the (symmetrized) $\chi_{-y}$ genus of the smooth quasi-projective variety $M:=M_{\PP^2}(r,n)$ is defined as 
\begin{align}\label{def:eqchiy}
\widehat{\chi}_{-y}(M) = y^{-\frac{\dim(M)}{2}} \chi(M,\Lambda_{-y} \Omega_M) \in \Q(t_1,t_2,e_0,\ldots,e_{r-1})[y^{\pm \frac{1}{2}}].\end{align}

Here, the right-hand side records the virtual $\T$-characters of \begin{align}\label{def:chiyloc}\chi(M,\Omega^k_M)=\sum_{i} (-1)^{i} H^i(M,\Omega^k_M).\end{align} By \cite[Lemma 4.2]{NY2}, the $\T$-weight spaces of all cohomology groups $H^i(M,\Omega^k_M)$ are finite dimensional. The $\T$-character of each $\chi(M,\Omega^k_M)$ is therefore well-defined as a Laurent series in the variables of $\T$.

Setting $\widehat{M}=M_{\widehat{\PP}^2}(r,\ell,n)$, the symmetrized $\chi_{-y}$-genus is defined as $$\widehat{\chi}_{-y}(\widehat{M})= y^{-\frac{\dim(\widehat{M})}{2}}\sum_{k} (-y)^{k} \chi(\widehat{M},\Omega^k_{\widehat{M}})\in \Q(t_1,t_2,e_0,\ldots,e_{r-1})[y^{\pm \frac{1}{2}}].$$ The $\T$-character $\chi(\widehat{M},\Omega^k_{\widehat{M}})$  is well-defined as a Laurent series by \cite[Theorem 3.4]{NY1} and \cite[Lemma 4.2]{NY2}. 

\subsection{$T_2$-fixed locus of $M$} 
To make contact with \eqref{eqn:defGS}, we write $\widehat{\chi}_{-y}(M)$ as a holomorphic Euler characteristic over its $T_2$-fixed loci. 

\subsubsection{}\label{sec:T2fixedM}

One has 
\begin{align*}
M_{\PP^2}(r,n)^{T_2}=\bigsqcup_{n_0+\cdots+n_{r-1}=n} \prod_{i=0}^{r-1} M_{\PP^2}(1,n_i)&=\bigsqcup_{n_0+\cdots+n_{r-1}=n} \prod_{i=0}^{r-1} \Hilb^{n_i}(\C^2)\\&=\bigsqcup_{|\bfn|=n}\Hilb^{\bfn}(\C^2).
\end{align*}

Let $\mathcal{F}$ denote the universal sheaf on $ \PP^2\times M_{\PP^2}(r,n)$. If $\pi$ is the projection onto the second factor, then the Kodaira-Spencer isomorphism implies that  \begin{align}\label{eq:TMrn} T_{M_{\PP^2}(r,n)}=R\hom_{\pi}(\mathcal{F},\mathcal{F}(-\ell_{\infty}))[1].\end{align}

We let $\mathcal{I}_i$ denote the pullback to $\C^2\times \Hilb^{\bfn}(\C^2)$ of the universal (ideal) sheaf on $\C^2\times \Hilb^{n_i}(\C^2)$. Considering $\C^2\times \Hilb^{\bfn}(\C^2)\subset \PP^2 \times M_{\PP^2}(r,n)$, we have $$\mathcal{F}|_{\C^2\times \Hilb^{\bfn}(\C^2)}\cong \bigoplus_{i=0}^{r-1} \mathcal{I}_i \cdot e_i. $$
Letting $\pi$ now denote the projection from $\C^2\times \Hilb^{\bfn}(\C^2)$ to the second factor, we conclude that 
\begin{align}\label{eq:TMres}T_{M_{\PP^2}(r,n)}|_{\Hilb^{\bfn}(\C^2)}=\bigoplus_{i,j=0}^{r-1}\big(R\pi_*\O -R\hom_{\pi}(\mathcal{I}_i,\mathcal{I}_j)\big)\cdot \frac{e_j}{e_i}\in K_{\T}^0(\Hilb^\bfn\C^2).\end{align}
The equality \eqref{eq:TMres} also follows from the description of $T_{M_{\PP^2}(r,n)}$ in terms of tautological bundles given, for example, in \cite[Equation (6.19)]{Neg}.

We set \begin{align}\label{def:V1def} V_1^{[\bfn]}:= T_{M_{\PP^2}(r,n)}|_{\Hilb^{\bfn}(\C^2)}. \end{align}
In particular, $$T_{\Hilb^{\bfn}(\C^2)}=\bigoplus_{i=0}^{r-1} \big( R\pi_*\O -R\hom_{\pi}(\mathcal{I}_i,\mathcal{I}_i) \big)=(V_1^{[\bfn]})^{f},$$ where $(-)^f$ denotes the $T_2$-fixed piece.

Set \begin{align}\label{def: Vdef}
V^{[\bfn]}=(1-y^{-1})\cdot V_1^{[\bfn]}\in K^0_{\mathbb{T}}(\Hilb^{\bfn}(\C^2)).\end{align} Applying K-theoretic equivariant localization (\cite[Theorem 3.5]{Thom}) to the action of $T_2$ on $M_{\PP^2}(r,n)$, we obtain \begin{align}\label{eq:chiyloc}\widehat{\chi}_{-y}(M_{\PP^2}(r,n))=\sum_{|{\bfn}|=n}\chi\Big(\Hilb^{\bfn}(\C^2), \frac{\Lambda_{-y}(\Omega_{\Hilb^{\bfn}(\C^2)})}{\Lambda_{-1}((V^{[\bfn]})^{\vee})^{m}}\otimes \det(V^{[\bfn]})^{-\frac{1}{2}}\Big),\end{align}
where $(-)^m$ denotes the $T_2$-moving part.
\newline

\subsection{Wall-crossing identities}

We state three identities for $\chi_{-y}(M_{\PP^2}(r,n))$ that will be applied in the next section to constrain Vafa-Witten invariants.

Let $\mathfrak{S}_r$ denote the symmetric group on $r$ elements $\{0, \ldots, r-1\}$.
\begin{theorem}\label{thm:symmNek}
We have
\begin{enumerate}
\item[$(1)$] $\widehat{\chi}_{-y}(M_{\PP^2}(r,n)) \Big|_{(t_1,t_2,e_{\sigma(0)}.\ldots, e_{\sigma(r-1)})} = \widehat{\chi}_{-y}(M_{\PP^2}(r,n))$ for all $\sigma \in \mathfrak{S}_r$,
\item[$(2)$] $\widehat{\chi}_{-y}(M_{\PP^2}(r,n)) \Big|_{(t_1,t_2,e_0^{-1},\ldots, e_{r-1}^{-1})} = \widehat{\chi}_{-y}(M_{\PP^2}(r,n))$.
\item[$(3)$] $\mathrm{[Kuhn\textrm{-}Leigh\textrm{-}Tanaka]}$
For any $\ell \in \Z$, we have
$$
\sum_{n} \widehat{\chi}_{-y}(M_{\widehat{\PP}^2}(r,\ell,n)) q^{n} =  \frac{\Theta_{A_{r-1},\ell}}{\overline{\eta}^r} \sum_{n} \widehat{\chi}_{-y}(M_{\PP^2}(r,n)) q^n.
$$
\end{enumerate}
\end{theorem}
\begin{proof}
By \eqref{def: Vdef}, the substitution  $$t_i\mapsto t_i, \  \ e_{j}\mapsto e_{\sigma(j)}$$ sends the $(n_0,\ldots,n_{r-1})$-summand of the right-hand side of \eqref{eq:chiyloc} to the \\ $(n_{\sigma^{-1}(0)},\ldots,n_{\sigma^{-1}(r-1)})$-summand. This substitution thus permutes the summands of the right-hand side of \eqref{eq:chiyloc}. Statement (1) follows.

The identity (2) is new and will be proved in Section \ref{sec:stcost}. In particular, in contrast to identity (1), we do not see how to obtain identity (2) from a bijection on components of some fixed locus of a torus action on ${M_{\PP^2}(r,n)}$ when $r>2$.  

The identity (3) is \cite[Theorem 1.4]{KLT}.  
\end{proof}

\subsection{$T_2$-fixed locus of $\widehat{M}$}

\subsubsection{}
We formulate an analog of \eqref{eq:chiyloc} for $\widehat{M}.$
Given a vector ${\boldsymbol{\ell}}=(\ell_0,\ldots,\ell_{r-1})\in \Z^{r}$ with $|\boldsymbol{\ell}|=\ell$, we set the constant  $$D_{\boldsymbol{\ell}}:= \frac{1}{2r}\sum_{i<j}(\ell_i-\ell_j)^2\in\Q.$$

As in Section \ref{sec:T2fixedM}, there is a decomposition 
\begin{align*}
M_{\widehat{\PP}^2}(r,\ell,n)^{T_2}&=\bigsqcup_{|\boldsymbol{\ell}|=\ell }\ \bigsqcup_{|\bfn|=n-D_{\boldsymbol{\ell}}} \prod_{i=0}^{r-1} M_{\widehat{\PP}^2}(1,\ell_i,n_i)\\&\cong \bigsqcup_{|\boldsymbol{\ell}|=\ell}\ \bigsqcup_{|\bfn|=n-D_{\boldsymbol{\ell}}} \Hilb^{\bfn}_{\bfell}(\widehat{\C}^2),
\end{align*}
where $\Hilb^{\bfn}_{\bfell}(\widehat{\C}^2):=\Hilb^{\bfn}(\widehat{\C}^2).$ Here the subscript $\bfell$ is used to index the component as well as keep track of the $\T$-action.

Let $\boldsymbol{\ell}, {\bfn}$ be such that  $|\boldsymbol{\ell}|=\ell$ and and $|\bfn|=n-D_{\boldsymbol{\ell}}$.  We let ${\mathcal{J}}_i$ denote the pullback of the universal sheaf on the $i$-th factor $\widehat{\C}^2\times \Hilb^{n_i}(\widehat{\C}^2).$

Set 
\begin{align*} W_{1,\bfell}^{[\bfn]}&:=T_{M_{\widehat{\PP}^2}(r,\ell,n)}|_{\Hilb^{\bfn}_{\boldsymbol{\ell}}(\widehat{\C}^2)}\\&=\bigoplus_{i,j=0}^{r-1} \big(R\pi_* \O -R\hom_{\pi}(\mathcal{J}_i\boxtimes\O(\ell_iC),\mathcal{J}_j\boxtimes\O(\ell_jC)\big) \big)\cdot \frac{e_j}{e_i}, \end{align*}

and 
\begin{align*}W^{[\bfn]}_{\bfell}&:=(1-y^{-1})\cdot W_{1,\bfell}^{[\bfn]}\in K^0_{\T}(\Hilb^{\bfn}_{\bfell}(\widehat{\C}^2)).\end{align*}
Then $$T_{\Hilb^{\bfn}_{\bfell}(\widehat{\C}^2)}=(W_{1,\boldsymbol{\bfell}}^{[\bfn]})^{f},$$ where $(-)^f$ denotes the $T_2$-fixed piece, and, as in \eqref{eq:chiyloc} we have 

\begin{align}\label{eq:chiyblloc}
\widehat{\chi}_{-y}(M_{\widehat{\PP}^2}(r,\ell,n))=\sum_{\begin{subarray}{c} |\bfell|=\ell, \\ |\bfn|=n-D_{\bfell}\end{subarray}}\chi\Big(\Hilb^{\bfn}_{\bfell}(\widehat{\C}^2), \frac{\Lambda_{-y}(\Omega_{\Hilb^{\bfn}_{\bfell}(\widehat{\C}^2)})}{\Lambda_{-1}\big((W^{[\bfn]}_{\bfell})^{\vee}\big)^{m}}\otimes \det(W^{[\bfn]}_{\bfell})^{-\frac{1}{2}}\Big),
\end{align}
where $(-)^{m}$ denotes the $T_2$-moving part. 

For any fixed $\ell$ and $n$, there are only finitely many $\bfell$ with $|\bfell|=\ell$ satisfying $D_{\bfell}\leq n.$ So, the sum on right-hand side of \eqref{eq:chiyblloc} is finite.

\subsubsection{}

This latter expression can be normalized along the lines of \eqref{def:T0[n]}. Set 
$$S_{{\boldsymbol\ell}}=T_{M_{\widehat{\PP}^2}(r,\ell,D_{\bfell})}|_{\Hilb^{\boldsymbol{0}}_{\boldsymbol{\ell}}(\widehat{\C}^2)} \in K^0_{\T}(\Hilb^{\boldsymbol{0}}_{\boldsymbol{\ell}}(\widehat{\C}^2))=K^0_{\mathbb{T}}(\pt)=\Z[t^{\pm 1}_1,t^{\pm 1}_2,e_i^{\pm 1},y^{\pm 1}].$$
Explicitly,
\begin{align*}S_{\boldsymbol{\ell}}&=\sum_{i,j=0}^{r-1} \chi\Big(\widehat{\C}^2, \O-\O\big((\ell_j-\ell_i)C\big)\Big) \cdot \frac{e_j}{e_i}.\end{align*}

There is a formula for $S_{\bfell}$ due to Nakajima-Yoshioka \cite[Theorem 3.6]{NY1}
\begin{align}\label{eqn:NYlbdef}
S_{\boldsymbol{\ell}} = \sum_{i,j=0}^{r-1} \frac{e_j}{e_i} \cdot \left\{\begin{array}{cc} \sum_{m,n \geq 0 \atop m+n \leq \ell_i -\ell_j - 1} t_1^{-m} t_2^{-n} & \mathrm{if \ } \ell_i - \ell_j > 0 \\  \sum_{m,n \geq 0 \atop m+n \leq  \ell_j -\ell_i - 2} t_1^{m+1} t_2^{n+1} & \mathrm{if \ } 
 \ell_i - \ell_j < -1 \\ 0 & \mathrm{otherwise.} \end{array} \right.
\end{align}
In particular, $$\mathrm{rk}(S_{\bfell})=\sum_{i<j}(\ell_i-\ell_j)^2=2rD_{\bfell}.$$

Along the lines of \eqref{def:Upsiloncstorigin}, we define 

\begin{align*}\Upsilon^{\circ}_{\bfell}&:=\chi\Big(\pt,\frac{\det(y^{1/2}S^{\vee}_{\bfell})}{\Lambda_{-1}(S_{\bfell}^{\vee}-y^{-1}S_{\bfell})}\Big)(-1)^{\sum_{i}(r-1)i(\ell_{i-1}-\ell_{i})^2}\\&=(-1)^{2rD_{\bfell}+\sum_{i}(r-i)i(\ell_{i-1}-\ell_{i})^2}y^{-rD_{\bfell}}\frac{\Lambda_{-1}(S_{\bfell}^{\vee}\otimes y)}{\Lambda_{-1} S_{\bfell}^{\vee}}\\&=y^{-rD_{\bfell}}\frac{\Lambda_{-y}S_{\bfell}^{\vee}}{\Lambda_{-1} S_{\bfell}^{\vee}}.\end{align*}

\subsubsection{}

We will use the following combinatorial lemma. 
\begin{lemma} \label{lem:comb}
Let $\boldsymbol{\ell} = (\ell_0, \ldots, \ell_{r-1}) \in \mathbb{Z}^r$ such that $\sum_i \ell_i = \ell$.
Then the specialization
\begin{align}\label{def:Sspec}
\Upsilon^{\circ}_{\bfell}|_{t_1=t_2=1,\ e_i=y^{-i}}
\end{align}
is well-defined. Moreover, the specialization \eqref{def:Sspec} is equal to zero unless $\epsilon_i:=\ell_{i-1} - \ell_{i} \in \{0,1\}$ for all $i = 1, \ldots, r-1$. When this is the case, we have
$$
\Upsilon^{\circ}_{\bfell}|_{t_1=t_2=1,\ e_i=y^{-i}} = \prod_{i=1}^{r-1} \binom{r}{i}^{\epsilon_i}_y \prod_{1 \leq i < j \leq r-1} \Bigg( \frac{[j]_y[r-i]_y}{[j-i]_y[r]_y} \Bigg)^{-\epsilon_i \epsilon_j}.
$$
\end{lemma}
\begin{proof}
No term of $S_{\bfell}$ is constant under the specialization $t_1=t_2=1$ and $e_i = y^{-i}$. So, the  specialization \eqref{def:Sspec} is well-defined. Note also that all terms in $S_{\boldsymbol{\ell}}$ come with positive multiplicity.

We consider the character $S_{\boldsymbol{\ell}}^\vee \otimes y$ appearing  in the numerator. Suppose there exists an index $i_0$ such that $\ell_{i_0-1} - \ell_{i_0} \leq -1$. Taking $i=i_0$ and $j=i_0-1$, we see that $S_{\boldsymbol{\ell}}^\vee \otimes y$ contains the term $1$, hence $\Lambda_{-y}S_{\boldsymbol{\ell}}^\vee  = 0$. 

Suppose there exists an index $i_0$ such that $\ell_{i_0-1} - \ell_{i_0} \geq 2$. Taking $i=i_0$ and $j=i_0-1$, we find that $S_{\boldsymbol{\ell}}^\vee \otimes y$ contains the term $t_1^{-1}t_2^{-1}$, hence $\Lambda_{-y}S_{\boldsymbol{\ell}}^\vee $ vanishes under the specialization $t_1=t_2=1$.

Now assume $\epsilon_i:=\ell_{i-1} - \ell_{i} \in \{0,1\}$ for all $i=1, \ldots r-1$. By \eqref{eqn:NYlbdef}, we compute 

\begin{align*}
&\frac{\Lambda_{-y}S_{\boldsymbol{\ell}}^\vee }{\Lambda_{-1}S_{\boldsymbol{\ell}}^{\vee}}\Big|_{t_1=t_2=1,\ e_i=y^{-i}}\\&=\prod_{0\leq i<j\leq r-1} \Big(\frac{1-y^{j-i+1}}{1-y^{j-i}}\Big)^{\binom{\ell_i-\ell_j+1}{2}}\Big(\frac{1-y^{i-j+1}}{1-y^{i-j}}\Big)^{\binom{\ell_j-\ell_i+1}{2}}
\\&=\prod_{0\leq i<j\leq r-1}\Big(y^{\frac{1}{2}}\frac{[j-i+1]_{y}}{[j-i]_y}\Big)^{\binom{\ell_i-\ell_j+1}{2}}\Big(y^{\frac{1}{2}}\frac{[j-i-1]_y}{[j-i]_y}\Big)^{\binom{\ell_j-\ell_i+1}{2}}
\\& = \prod_{0\leq i<j\leq r-1} \Big(y\frac{[j-i+1]_y[j-i-1]_y}{[j-i]_y^2}\Big)^{\frac{(\ell_i-\ell_j)^2}{2}}\Big(\frac{[j-i+1]_y}{[j-i-1]_y}\Big)^{\frac{\ell_i-\ell_j}{2}}.
\end{align*} 

Using that $$\ell_i-\ell_j=\epsilon_{i+1}+\cdots+\epsilon_{j}$$ for $0\leq i<j<r$  and $\epsilon_i^2=\epsilon_i$ when $\epsilon_i\in \{0,1\},$ we find that 

\begin{align*}
&y^{-\frac{1}{2} \sum_{i<j} (\ell_i - \ell_j)^2 } \frac{\Lambda_{-y} S_{\boldsymbol{\ell}}^\vee }{\Lambda_{-1} S_{\boldsymbol{\ell}}^\vee  }\Big|_{t_1=t_2=1,\ e_i=y^{-i}} 
\\&=\prod_{i=1}^{r-1}\prod_{\begin{subarray}{c} 0\leq i'<i \\ i\leq j' < r\end{subarray}} \Big(\frac{[j'-i'+1]_y}{[j'-i']_y}\Big)^{\epsilon_i} \cdot  \prod_{0\leq i < j < r} \prod_{\begin{subarray}{c} 0\leq i'<i \\ j\leq j'< r\end{subarray}}\Big(\frac{[j'-i'+1]_y[j'-i'-1]_y}{[j'-i']_y^2}\Big)^{\epsilon_i\epsilon_j} \\& = \prod_{i=1}^{r-1}\binom{r}{i}_y^{\epsilon_i}\cdot \prod_{1\leq i< j<r} \Big(\frac{[r-1+1]_y[j-(i-1)-1]_y}{[j-0]_y[(r-1)-(i-1)]_y}\Big)^{\epsilon_i\epsilon_j},
\end{align*}
completing the proof.
\end{proof}
{Lemma \ref{lem:comb} implies that \eqref{def:Sspec} must vanish unless $$(\epsilon_i[C])^2=(\epsilon_i[C])\cdot K_{\widehat{\C}^2}\in H_2(\widehat{\C}^2,\Z)$$ for all $i$. Note the similarity to the condition $\beta^2=\beta K_S$ of \cite[Proposition 6.3.1]{Moc}.}

\section{From framed sheaves to Vafa-Witten invariants}\label{sec:framedtoVW}

\subsection{$\mathsf{G}_{\C^2,(c_1(L_{i-1}/L_{i}))}$ and $\mathsf{G}_{\widehat{\C}^2,((\ell_{i-1}-\ell_i)[C])}$ in terms of framed sheaves}

\subsubsection{}

We now prove Proposition \ref{prop:main} of the introduction.
\begin{proof} [Proof of Proposition \ref{prop:main}]
Consider a maximal $T$-invariant affine open subset $U_{\alpha}\cong \C^2$ of $S$.
We write $(w_1)_{\alpha}$ and $(w_2)_{\alpha}$ for the weights of the $T$-action on the coordinate directions of $U_{\alpha}$ and $L_i^{(\alpha)}$ for the $T$-character corresponding to the restriction $L_i|_{U_\alpha}$. 

We consider \eqref{def:T[n]} and \eqref{def:E[n]} for $S = U_\alpha$, so that $$\E^{[\bfn]}=\bigoplus_{i=0}^{r-1}\mathcal{I}_i\otimes L_i^{(\alpha)}y^{-i},$$ and
$$T^{[\bfn]} = (R\hom_{\pi}(\mathbb{E}^{[\bfn]}, \mathbb{E}^{[\bfn]} )_0)^{\vee} \otimes y - R\hom_{\pi}(\mathbb{E}^{[\bfn]}, \mathbb{E}^{[\bfn]})_0,
$$
where we have used Serre duality to rewrite the first summand.

Comparing \eqref{def:T0[n]} and \eqref{def: Vdef}, we compute
\begin{align}\label{eq:TfromM} T^{[\bfn]}_0&=T^{[\bfn]}-T^{[\bf 0]} \\ \nonumber & = \Big(-y\cdot(V_1^{[\bfn]})^{\vee} +V_1^{[\bfn]}\Big)\Big|_{t_i=(w_\alpha)_{i},\ e_j=L_{j}^{(\alpha)}y^{-j}}\in K^0_{T\times \C^*_y}(\Hilb^{\bfn} (U_{\alpha})),
\end{align}
where the substitution denotes taking the image under the induced map $$K^0_{\T}(\Hilb^{\bfn}(\C^2))\to K^0_{T\times \C^*_y}(\Hilb^{\bfn}(U_{\alpha})).$$
Note that the subtraction of $T^{[\bf 0]}$ corresponds to the removal of trace appearing in each summand of \eqref{def:V1def}. 
Recall from \eqref{def:altUps} that $$\Upsilon_{U_\alpha,\bfa,\bfn}=\chi\Big(\Hilb^{\bfn}(U_{\alpha}),\Lambda_{-1}(-\Omega_0^{[\bfn]}+\Omega_{\Hilb^{\bfn}(U_{\alpha})})\otimes \det(\Omega_0^{[\bfn]})^{\frac{1}{2}}\Big) \in \Q(y^{\frac{1}{2}},t_1,t_2).$$
Then by \eqref{eq:TfromM}, we have 

\begin{align}\label{eq:UpsfromM} \Lambda_{-1}&(\Omega_{\Hilb^{\bfn}(U_{\alpha})})\cdot\frac{\det(\Omega_0^{[\bfn]})^{\frac{1}{2}}}{\Lambda_{-1}(\Omega_0^{[\bfn]})}\\ \nonumber&=\Lambda_{-1}(\Omega_{\Hilb^{\bfn}(\C^2)})\cdot\frac{\big(\det(V_1^{[\bfn]})\cdot\det(y^{-1}\cdot V_1^{[\bfn]})\big)^{-\frac{1}{2}}}{\Lambda_{-1}\big((V_1^{[\bfn]})^{\vee}-y^{-1}\cdot V_1^{[\bfn]}\big)} \Big|_{t_i=(w_{\alpha})_i,\  e_j=L_j^{(\alpha)}y^{-j}} 
\\\nonumber &=(-1)^{\mathrm{rk}(V_1^{[\bfn]})}\Lambda_{-1}(\Omega_{\Hilb^{\bfn}(\C^2)})\cdot\frac{\big(\det(V_1^{[\bfn]})/\det(y^{-1}\cdot V_1^{[\bfn]})\big)^{-\frac{1}{2}}}{\Lambda_{-1}\big((V^{[\bfn]})^{\vee}\big)}\Big|_{t_i=(w_{\alpha})_i,\ e_j=L_j^{(\alpha)}y^{-j}}
\\ \nonumber &=\Bigg({\Lambda_{-y}(\Omega_{\Hilb^{\bfn}(\C^2)})}\cdot\frac{\big(\det(V^{[\bfn]})\big)^{-\frac{1}{2}}}{\Lambda_{-1}\big(((V^{[\bfn]})^{\vee})^{m}\big)}\Bigg)\Big|_{t_i=(w_{\alpha})_i,\ e_j=L_j^{(\alpha)}y^{-j}}\\ \nonumber & \qquad  \in K^0_{T\times \C^*_y}(\Hilb^{\bfn}(U_{\alpha}))\Big[y^{\frac{1}{2}},\frac{1}{1-y^{n}t_1^{k_1}t_2^{k_2}}\Big]_{n \in \Z_{\neq 0},\ k_i \in \Z}.
 \end{align}

In particular, $\Upsilon_{U_\alpha,\bfa,\bfn}$ matches the $\bfn$-summand of the right-hand side of    \eqref{eq:chiyloc} under the specialization $e_i=L^{(\alpha)}_iy^{-i}$. We conclude that 
\begin{align}
\label{eq:Gfromchi}
\sum_{n=0}^{\infty} q^n\cdot \widehat{\chi}_{-y}(M_{\mathbb{P}^2}(r,n))\Big|_{((w_\alpha)_1, (w_\alpha)_2,L_0^{(\alpha)},L_1^{(\alpha)}y^{-1}, \ldots, L_{r-1}^{(\alpha)}y^{-(r-1)})}  = \mathsf{G}_{U_{\alpha},(c_1(L^{(\alpha)}_{i-1}/L^{(\alpha)}_{i}))}.
\end{align}
The proposition then follows from \eqref{eqn:final} and Lemma \ref{lem:prod}. \end{proof}

We remark that each term of \eqref{eq:Gfromchi} should be regarded as a rational function.  Although $\widehat{\chi}_{-y}(M_{\mathbb{P}^2}(r,n))$ is a well-defined Laurent series, its specialization on the left-hand side of \eqref{eq:Gfromchi} need not be well-defined as a Laurent series for arbitrary choices of $L_{i}^{(\alpha)}$. Rather  $\widehat{\chi}_{-y}(M_{\mathbb{P}^2}(r,n))$ can be written (using, for example, localization) as a rational function in $t_i, e_j$ that is well-defined under the specialization from \eqref{eq:Gfromchi}.

\begin{remark} \label{rem:propmainstrong}
Referring to the universal functions $\overline{E}_i$, $\overline{E}_{ij}$ defined in \eqref{eqn:univdef}, we in fact proved something stronger. Namely for any smooth projective toric surface $S$ and \emph{any} $T$-equivariant algebraic classes $\bfa  = (a_1, \ldots, a_{r-1})\in H^2_T(S,\Z)^{r-1}$, we have
\[
A^{\chi(\O_S)} B^{K_S^2} \prod_i \overline{E}_{i}^{a_i K_S} \prod_{i \leq j} \overline{E}_{ij}^{a_i a_j} = \Bigg(\prod_{\alpha=1}^{e(S)} \sum_{n=0}^{\infty} q^n \widehat{\chi}_{-y}(M_{\PP^2}(r,n)) \Big|_{(w_\alpha,\{L_i^{(\alpha)}y^{-i}\}_i)}\Bigg)_{t_1=1,t_2=1}.
\]
In particular, the product in the parentheses on the right-hand side is a Laurent polynomial in $t_1,t_2$, and so is well-defined under the specialization $t_1=1,t_2=1$. \end{remark}

\subsubsection{}
Equip $\widehat{\C}^2$ with the $T$-action induced from the action on $\C^2$ scaling the coordinate directions with weights $t_1$ and $t_2$, see Figure \ref{fig}.  We now relate the terms of the form $\mathsf{G}_{\widehat{\C}^2,((\ell_{i-1}-\ell_i)[C])}$ to $\widehat{\chi}_{-y}(M_{\widehat{\PP}^2}(r,\ell,n)).$ The argument is similar to that of the previous section. Namely, let $$ \Hilb^{\bfn}_{{\bfell}}(\widehat{\C}^2)=\prod_{i=0}^{r-1}M_{\widehat{\PP}^2}(1,\ell_i,n_i)\subset M_{\widehat{\PP}^2}(r,\ell,n)^{T_2}$$ with $|\bfell|=\ell$ and $|\bfn|=n-D_{\bfell}$ be a component of the $T_2$-fixed locus, and consider \eqref{def:T[n]} and \eqref{def:E[n]} for $S=\widehat{\C}^2$ with $L_i=\O(\ell_iC)$ for $\ell_i\in \Z$, so that $$\E^{[\bfn]}=\bigoplus_{i=0}^{r-1}\mathcal{J}_i\boxtimes\mathcal{O}(\ell_iC) y^{-i}$$
and 
$$T^{[\bfn]} = (R\hom_{\pi}(\mathbb{E}^{[\bfn]}, \mathbb{E}^{[\bfn]} )_0)^{\vee} \otimes y - R\hom_{\pi}(\mathbb{E}^{[\bfn]}, \mathbb{E}^{[\bfn]})_0.$$
Arguing as in \eqref{eq:TfromM}, one has 
\begin{align*}
T^{[\bfn]}_0&=T^{[\bfn]}-T^{[\bf0]}
\\& =\Big(-y\cdot\big(W_{1,\bfell}^{[\bfn]}-S_{\bfell}\big)^{\vee}+W^{[\bfn]}_{1,\bfell}-S_{\bfell}\Big)\Big|_{e_i=y^{-i}} \in K^0_{T\times \C^*_y}(\Hilb^{\bfn}(\widehat{\C}^2)),
\end{align*}
where the substitution of variables denotes the corresponding induced map $$K^0_{\mathbb{T}}(\Hilb^{\bfn}_{\bfell}(\widehat{\C}^2))\to K^0_{T\times \C_y^*}(\Hilb^{\bfn}(\widehat{\C}^2)).$$

Arguing as in \eqref{eq:UpsfromM}, we obtain
\begin{align*}
\Lambda_{-1}&(\Omega_{\Hilb^{\bfn}(\widehat{\C}^2)})\cdot\frac{\det(\Omega_0^{[\bfn]})^{\frac{1}{2}}}{\Lambda_{-1}(\Omega_0^{[\bfn]})}\\&= \Big({\Lambda_{-1}(\Omega_{\Hilb^{\bfn}_{\bfell}(\widehat{\C}^2)})}\cdot\frac{\det(y^{\frac{1}{2}}S^{\vee}_{\bfell})^{-1}}{\Lambda_{-1}(y^{-1}S_{\bfell}-S_{\bfell}^{\vee})}
\frac{\det(y^{-\frac{1}{2}}\cdot W_{1,\bfell}^{[\bfn]})^{-1}}{\Lambda_{-1}\big((W_{1,\bfell}^{[\bfn]})^{\vee}-y^{-1}\cdot W_{1,\bfell}^{[\bfn]}\big)} \Big)\Big|_{e_i=y^{-i}}
\\&= \Big({(-1)^{2r(D_{\bfell}+n)}}{} \frac{\Lambda_{-y}(\Omega_{\Hilb^{\bfn}_{\bfell}(\widehat{\C}^2)})}{\Upsilon^{\circ}_{\bfell}}\frac{\big(\det(W^{[\bfn]}_{\bfell})\big)^{-\frac{1}{2}}}{\Lambda_{-1}((W^{[\bfn]}_{\bfell})^{\vee})^{m}}\Big)\Big|_{ e_i=y^{-i}}\\ & \qquad \in K^0_{T\times \C^*_y}(\Hilb^{\bfn}(\widehat{\C}^2))\Big[y^{\frac{1}{2}},\frac{1}{1-y^{n}t_1^{k_1}t_2^{k_2}}\Big]_{n \in \Z_{\neq 0},\ k_i \in \Z}.
\end{align*}

As $(-1)^{2rn}=(-1)^{2rD_{\bfell}}$,
we conclude that 
\begin{align}
\label{eq:GfromchiBl} 
\sum_{n}q^{n}\cdot\widehat{\chi}_{-y}(M_{\widehat{\PP}^2}(r,\ell,n))\Big|_{e_i=y^{-i}}&=\sum_{|\bfell|=\ell }q^{D_{\bfell}}\cdot \Upsilon^{\circ}_{\bfell}\big|_{e_i=y^{-i}}\cdot \mathsf{G}_{\widehat{\C}^2,((\ell_{i-1}-\ell_i)[C])}.
\end{align}

Only finitely many terms of the sum right-hand side of \eqref{eq:GfromchiBl} contribute to any given term of the left-hand sum.

\subsection{Symmetry equations}
We prove the first set of equations of Theorem  \ref{thm:main}, which we refer to as \emph{symmetry equations}.
\begin{proof}[Proof of Theorem \ref{thm:main}, symmetry equations]
We prove the stronger statement that
$$
\overline{E}_{r-i} = \overline{E}_i, \quad \overline{E}_{r-j,r-i} = \overline{E}_{ij}.
$$
Note that by \eqref{eqn:defbarCij} and \eqref{eqn:defCij}, these equalities imply that $C_{ij} = C_{r-j,r-i}$ for all $i \leq j$.

Let $S$ be a smooth projective toric surface, $\bfa  = (a_1, \ldots, a_{r-1})\in H^2_T(S,\Z)^{r-1}$ be a choice of $T$-equivariant algebraic classes and consider the equation in Remark \ref{rem:propmainstrong}. Then for any $\alpha, n$, denoting $M:=M_{\mathbb{P}^2}(r,n)$, we have
\begin{align*}
\widehat{\chi}_{-y}(M) \Big|_{(w_\alpha,L_0^{(\alpha)}, L_1^{(\alpha)}y^{-1}, \ldots, L_{r-1}^{(\alpha)}y^{-(r-1)})} &= \widehat{\chi}_{-y}(M) \Big|_{(w_\alpha,L_{r-1}^{(\alpha)} y^{-(r-1)}, L_{r-2}^{(\alpha)}y^{-(r-2)}, \ldots, L_{0}^{(\alpha)})} \\
&=\widehat{\chi}_{-y}(M) \Big|_{(w_\alpha,(L_{r-1}^{(\alpha)})^{-1} y^{r-1}, (L_{r-2}^{(\alpha)})^{-1}y^{r-2}, \ldots, (L_{0}^{(\alpha)})^{-1})} \\
 &=\widehat{\chi}_{-y}(M) \Big|_{(w_\alpha,(L_{r-1}^{(\alpha)})^{-1}, (L_{r-2}^{(\alpha)}y)^{-1}, \ldots, (L_{0}^{(\alpha)}y^{(r-1)})^{-1})},
\end{align*}
where we used Theorem \ref{thm:symmNek} (1) for the first equality, Theorem \ref{thm:symmNek} (2) for the second equality, and the fact that the parameters $e_i$ only enter these expressions through their quotients $e_i^{-1} e_j$, which follows from \eqref{eq:TMres}, for the third equality. Defining $b_i := a_{r-i}$ for all $i=1, \ldots, r-1$, 
we obtain
\begin{align*}
\overline{A}^{\chi(\O_S)} \overline{B}^{K_S^2} \prod_i \overline{E}_{i}^{a_i K_S} \prod_{i \leq j} \overline{E}_{ij}^{a_i a_j} &= \overline{A}^{\chi(\O_S)} \overline{B}^{K_S^2} \prod_i \overline{E}_{i}^{b_i K_S} \prod_{i \leq j} \overline{E}_{ij}^{b_i b_j} \\
&= \overline{A}^{\chi(\O_S)} \overline{B}^{K_S^2} \prod_i \overline{E}_{r-i}^{a_i K_S} \prod_{i \leq j} \overline{E}_{r-j,r-i}^{a_i a_j}.
\end{align*}
As this equation holds for all choices of $S, a_i$, and $\overline{E}_{i}, \overline{E}_{ij}$ have constant coefficient $1$, the result follows.
\end{proof}

\subsection{Blow-up equations}
We prove the second set of equations of Theorem  \ref{thm:main}, which we refer to as \emph{blow-up equations}.
\begin{proof}[Proof of Theorem \ref{thm:main}, blow-up equations]
Fix a smooth projective toric surface $S$. Denote by $\pi \colon \widehat{S} \to S$ the blow-up in one $T$-fixed point with exceptional divisor $C$. 

Let $e:=e(S)$ and let $\{U_\alpha\}_{\alpha=1}^e$ be the cover by maximal $T$-invariant affine open subsets of $S$. We denote the cover by maximal $T$-invariant affine open subsets of $\widehat{S}$ by $\{U'_\alpha\}_{\alpha=1}^{e+1}$, where the labeling is chosen such that $U_1', U_{e+1}'$ contain the exceptional divisor $C$. In particular $U_\alpha' \cong U_\alpha$ for all $\alpha \neq 1,e+1$. Choose the $T$-action so that the coordinate directions of $U_1$ are scaled by $t_1$ and $t_2$. Then, without loss of generality, the coordinate directions of $U_{1}'$ are scaled by $t_1$ and  $t_1^{-1}t_2$ and the coordinate directions of $U_{e+1}'$ are scaled by $t_1 t_2^{-1}$ and $t_2$, see Figure \ref{fig}.

Fix $\ell\in \Z$. Pick $\epsilon_{i}\in \{0,1\}$ for $i=1,\ldots,r-1$. We have $$(\epsilon_{i} [C])\cdot K_{\widehat{S}}=(\epsilon_{i}[C])^2=-\epsilon_i, \ \ \chi(\mathcal{O}_S)=\chi(\mathcal{O}_{\widehat{S}}), \ \ K_{\widehat{S}}^2=K_S^2-1.$$ Applying \eqref{eqn:univdef} to $(\widehat{S},(\epsilon_i[C]))$ and $(S,{\bf0})$ we find that $$B^{-1}\prod_{i\leq j}C_{ij}^{-\epsilon_i\epsilon_j}{\mathsf{G}_{S,{\bf0}}}=q^{-\frac{r}{24}+Q\big((\epsilon_i[C])\big)}\frac{\Upsilon_{\widehat{S},(\epsilon_i[C])}}{\Upsilon_{S,{\bf 0}}}\cdot \mathsf{G}_{\widehat{S},(\epsilon_i[C])}.$$

For any $\bfell$ satisfying $\ell_{i-1}-\ell_{i}=\epsilon_{i}$ for all $i$, from \eqref{def:Q} we have $$Q\big((\epsilon_{i}[C])\big)=D_{\bfell},$$ and by  \eqref{def:Upsiloncst} and Lemma \ref{lem:comb} we have \begin{align*}\frac{\Upsilon_{\widehat{S},(\epsilon_i[C])}}{\Upsilon_{S,{\bf0}}}&= \prod_{i=1}^{r-1}\binom{r}{i}^{\epsilon_{i}}_y\prod_{1\leq i<j\leq r-1}\Big(\frac{[j]_y[r-i]_y}{[j-i]_y[r]_y}\Big)^{-\epsilon_i\epsilon_j}
=\Upsilon_{\bfell}^{\circ}|_{t_1=t_2=1,\ e_i=y^{-i}}.
\end{align*}

Fix $\ell\in \{0,\ldots, r-1\}$. Note that there is a natural bijection between sequences $\bfell = (\ell_0, \ldots, \ell_{r-1}) \in \Z^{r}
$ with $\sum_{i=0}^{r-1} \ell_i = \ell$ 
 such that $\ell_{i-1} - \ell_i \in \{0,1\}$ for all $i\in [r-1]$ and subsets $I \subset [r-1]$ satisfying $|\!|I|\!| \equiv \ell \mod r$, where $|\!|I|\!|:=\sum_{i \in I} i$. 
 
 So, for any $I\subset [r-1]$ with $|\!|I|\!| \equiv \ell \mod r$, we may let $\bfell_{I}$ denote the unique sequence $\bfell=(\ell_0,\ldots,\ell_{r-1}) \in \Z^r$ such that $\sum_{i=0}^{r-1} \ell_i=\ell$ and for all $i\in [r-1]$ one has $$\ell_{i-1}-\ell_i= \begin{cases} 1 \hspace{.5cm} i\in I \\ 0 \hspace{.5cm} i\not\in I \end{cases}.$$ 
Recalling the definition \eqref{def:CI} of the universal series $C_{I}$, we have 
\begin{align*}\sum_{\begin{subarray}{c} I\subset [r-1]\\ |\!|I|\!| \equiv \ell \mod r\end{subarray}}C_{I}^{-1}\mathsf{G}_{S,{\bf0}}&=\sum_{\begin{subarray}{c} I\subset [r-1]\\ |\!|I|\!|\equiv \ell \mod r \end{subarray}} q^{-\frac{r}{24}+D_{\bfell_I}} \Upsilon_{\bfell_I}^{\circ}\Big|_{t_1=t_2=1,\ e_i=y^{-i}}\mathsf{G}_{\widehat{S},((\ell_{i-1}-\ell_{i})[C])}
\\&=\sum_{\begin{subarray}{c}\bfell \in \Z^r \\ \ell_0+\cdots+\ell_{r-1}= \ell \end{subarray}} q^{-\frac{r}{24}+D_{\bfell}} \Upsilon_{\bfell}^{\circ}\Big|_{t_1=t_2=1,\ e_i=y^{-i}}\mathsf{G}_{\widehat{S},((\ell_{i-1}-\ell_{i})[C])},\end{align*}
 where the latter equality holds by Lemma  \ref{lem:comb}, which implies that the $\bfell$-summand on the right-hand side vanishes unless $\ell_{i-1} - \ell_i \in \{0,1\}$ for all $i\in [r-1]$.

By \eqref{eq:GBlfact} and Lemma \ref{lem:prod}, $${\mathsf{G}_{\widehat{S},((\ell_{i-1}-\ell_{i})[C])}}=\Big(\mathsf{G}_{\widehat{\C}^2,((\ell_{i-1}-\ell_i)[C])}\displaystyle\prod_{\alpha=2}^{e}\mathsf{G}_{U_{\alpha},{\bf0}}\Big)\Big|_{t_1=t_2=1} $$ 
Therefore, by \eqref{eq:GfromchiBl}, 
\begin{align*}
    \sum_{\begin{subarray}{c} I\subset [r-1]\\ |\!|I|\!|\equiv \ell \mod r \end{subarray}}C_{I}^{-1}\mathsf{G}_{S,{\bf0}}
    &=\sum_{\begin{subarray}{c}\bfell \in \Z^r \\ 
 \ell_0+\cdots+\ell_{r-1}= \ell \end{subarray}} \Big( q^{-\frac{r}{24}+D_{\bfell}} \Upsilon_{\bfell}^{\circ}\Big|_{ e_i=y^{-i}} \mathsf{G}_{\widehat{\C}^2,((\ell_{i-1}-\ell_i)[C])}\cdot\displaystyle\prod_{\alpha=2}^{e}\mathsf{G}_{U_{\alpha},{\bf0}}\Big)\Big|_{{t_1=t_2=1}}
    \\&=  \Big(q^{-\frac{r}{24}}\sum_{n} q^n \widehat{\chi}_{-y}(M_{\widehat{\PP}^2}(r,\ell,n))
\big|_{e_i=y^{-i}}\prod_{\alpha=2}^{e}\mathsf{G}_{U_{\alpha},{\bf0}}\Big)\Big|_{{t_1=t_2=1}}.
\end{align*}
Applying Kuhn-Leigh-Tanaka's blow-up formula (Theorem \ref{thm:symmNek}(3)), \eqref{eq:Gfromchi}, and Lemma \ref{lem:prod} again yields

 \begin{align*}
   \sum_{\begin{subarray}{c} I\subset [r-1]\\ |\!|I|\!|\equiv\ell \mod r \end{subarray}}C_{I}^{-1}\mathsf{G}_{S,{\bf0}}
    &=  
\Big(\frac{\Theta_{A_{r-1},\ell}}{{\eta}^r} \sum_{n} q^n\widehat{\chi}_{-y}(M_{\PP^2}(r,n))\big|_{e_i=y^{-i}}\prod_{\alpha=2}^{e}\mathsf{G}_{U_{\alpha,{\bf0}}} \Big)\Big|_{t_1=t_2=1}\\&= \frac{\Theta_{A_{r-1},\ell}}{{\eta}^r} \Big(\prod_{\alpha=1}^{e}\mathsf{G}_{U_{\alpha,{\bf0}}} \Big)\Big|_{t_1=t_2=1} 
 \\& = \frac{\Theta_{A_{r-1},\ell}}{{\eta}^r} \mathsf{G}_{S,{\bf0}}.
\end{align*}
We deduce that $$\sum_{\begin{subarray}{c} I\subset [r-1]\\ |\!|I|\!|\equiv \ell \mod r \end{subarray}}C_{I}^{-1} =\frac{\Theta_{A_{r-1},\ell}}{{\eta}^r}. $$

The second set of equations of Theorem \ref{thm:main} then follows  
from the theta function identity (cf.~\cite[Equation (39)]{GKL})
\begin{equation*} 
\sum_{\ell = 0}^{r-1} \epsilon_r^{k \ell} \Theta_{A_{r-1},\ell} = \Theta_{A_{r-1}^\vee,k}. \qedhere
\end{equation*}
\end{proof}

\section{Stable/co-stable wall-crossing}
In this section, we prove identity (2) of Theorem \ref{thm:symmNek}. We deduce the result from the invariance of the equivariant $\chi_{-y}$-genus of $M_{\PP^2}(r,n)$ under a certain GIT wall-crossing.

\subsection{Quiver presentation of $M$ } \label{sec:quiver}

We first recall the presentation of $M$ as a Nakajima quiver variety associated to the Jordan quiver $\widehat{A}_0$. Consider the linear space $$\widetilde{M}(r,n)= \Hom(\C^n,\C^n)^{\oplus 2}\oplus \Hom(\C^r,\C^n)\oplus \Hom(\C^n,\C^r)$$ of fixed-dimensional representations of the framed and doubled Jordan quiver. 

The torus $\T$ acts on $\widetilde{M}(r,n)$ as follows. An element $(t_1,t_2,e_0,\ldots,e_{r-1},y)\in \T$ sends   $$(X_1,X_2,I,J)\mapsto {\Big(} t_1X_1,t_2X_2,I \mathrm{diag}(e_0,\ldots,e_{r-1})^{-1},t_1t_2\cdot \mathrm{diag}(e_0,\ldots,e_{r-1})J  {\Big)},$$ 
and the group $\GL(n)$ acts on  $\widetilde{M}(r,n)$ by $$g\cdot (X_1,X_2,I,J)= (gX_1g^{-1}, gX_2g^{-1}, gI, Jg^{-1}).$$ The actions of $\T$ and $\GL(n)$ commute. 

There is a $\T$-equivariant isomorphism of varieties between $M_{\PP^2}(r,n)$ and the free quotient $$M(r,n):=\Big\{(X_1,X_2,I,J)\in \widetilde{M}(r,n) \, \Big | \begin{array}{l}[X_1,X_2]+IJ=0 \\ \C\langle X_1,X_2\rangle\mathrm{im}(I)=\C^n \end{array} \Big\}/ \GL(n).$$ 
This quotient is the Nakajima quiver variety associated to the Jordan quiver variety and GIT stability given by the character $\det \colon \GL(n)\to \C^*$. The Nakajima quiver variety associated to the identical quiver with opposite GIT stability condition given by the character $\det^{-1}\colon \GL(n)\to \C^*$ is $$M^c(r,n) := \Big\{(X_1,X_2,I,J)\in \widetilde{M}(r,n) \, \Big | \begin{array}{l}[X_1,X_2]+IJ=0 \\ J\big(\C\langle X_1,X_2\rangle v\big)\neq 0 \ \mathrm{for}\ 0\neq v\in \ker J\ \end{array} \Big\}/\GL(n).$$
Both $M(r,n)$ and $M^c(r,n)$ are $\T$-equivariant resolutions of the affine quotient $$M^0(r,n):=\Big\{(X_1,X_2,I,J)\in \widetilde{M}(r,n) \, \big | \, [X_1,X_2]+IJ=0 \Big\}/\!/\GL(n);$$
see, for example, \cite[Section 2]{NY2}. By \cite[Theorem 1.1]{CB1} and \cite[Proposition 2.5]{BL}, there are $\T$-equivariant isomorphisms between $M^0$ and each of $\Spec \Gamma(\O_{M(r,n)})$ and  $\Spec \Gamma(\O_{M^c(r,n)})$ such that the natural projective morphisms $$M(r,n)\to M^0(r,n),\ \ \ \ M^c(r,n)\to M^0(r,n)$$ given by variation of GIT stability coincide with the affinization maps. By \cite[Section 5.2, Proposition 5.2]{Cho}, the affine quotient $M^0(r,n)$ is reduced and normal.

The automorphism $\sigma\colon \widetilde{M}(r,n)\to \widetilde{M}(r,n)$ sending a quadruple $$(X_1,X_2,I,J)\mapsto (X_1^{t},X_2^{t}, -J^{t}, I^{t})$$ induces an isomorphism of varieties  $M(r,n)\rightarrow M^c(r,n)$, which we also denote by $\sigma$.
However, the induced isomorphism $\sigma$ is not $\T$-equivariant. Instead, one has \begin{align*}\sigma((t_1,t_2,e_0,\ldots,e_{r-1},y)[X_1,X_2,I,J])&=(t_1,t_2,\frac{e_0^{-1}}{t_1t_2},\ldots,\frac{e_{r-1}^{-1}}{t_1t_2},y)\sigma([X_1,X_2,I,J])\\&=(t_1,t_2,e_0^{-1},\ldots,e_{r-1}^{-1},y)\sigma([X_1,X_2,I,J]),\end{align*} where the last equality holds because the $1$-dimensional subtorus $\{(\mathrm{Id},c\cdot \mathrm{Id})\}$ in $ T_1\times T_2$ acts trivially on $M(r,n)$.

It follows that for any $r,n,k\geq 0$, there is an equality of $\T$-equivariant Euler characteristics $$\chi(M(r,n),\Omega^k_{M(r,n)} )_{e_i\mapsto e_i^{-1}} =\chi(M^{c}(r,n), \Omega^k_{M^{c}(r,n)})\in \Q(t_1,t_2,e_0,\ldots,e_{r-1}).$$

\subsection{Stable/co-stable wall crossing}\label{sec:stcost}

The equality $(2)$ of Theorem \ref{thm:symmNek} is therefore equivalent to Theorem \ref{thm:stcostintro},
which we recall asserts that
\begin{align*}
\chi(M(r,n),\Omega^k_{M(r,n)} ) =\chi(M^{c}(r,n), \Omega^k_{M^{c}(r,n)})\in \Q(t_1,t_2,e_0,\ldots,e_{r-1}). \end{align*}

\begin{remark} \label{rmk:generalNak} Theorem \ref{thm:stcostintro} generalizes to any smooth Nakajima quiver varieties as follows. Adopting the notation of \cite[Section 3.2]{BD}, let ${\mathbf N}_Q^{\zeta}({\bf f},{\bf d})$ and ${\mathbf N}^{\zeta'}_Q({\bf f},{\bf d})$ denote the pair of smooth Nakajima quiver varieties arising from two choices $\zeta, \zeta'$ of generic King stability conditions for fixed underlying quiver $Q$ and dimension vectors $\bf{f}$ and $\bf{d}$. Let $T$ denote the torus defined in \cite[Section 2.6]{BD} that acts on ${\mathbf N}_Q^{\zeta}({\bf f},{\bf d})$.
Our proof of Theorem \ref{thm:stcostintro} in Section \ref{sec:MHMproof} extends as written to show that $$\chi({\mathbf N}_Q^{\zeta}({\bf f},{\bf d}), \Omega^k)=\chi({\mathbf N}_Q^{\zeta'}({\bf f},{\bf d}), \Omega^k)\in \Q(T)$$ for all $k$. In this general setting, the role of the affine quotient $M^0(r,n)$ is played by the reduced scheme underlying ${\mathbf N}^{(0,\ldots,0)}_Q({\bf f},{\bf d})\cong \mathcal{M}_{({\bf d},1)}(\Pi_{Q_{\bf f}})$.

For example, when $Q$ is the $A_1$ quiver, one obtains an equality $$\chi(T^*\textrm{Gr}(n,r),\Omega^k)=\chi(T^*\textrm{Gr}(n-r,r),\Omega^k)\in \Q(\C^*\times (\C^*)^r),$$ where the first factor $\C^*$ scales the cotangent directions, and the action of the second factor $(\C^*)^r$ is induced from its action on $\C^r$.

To simplify the notation, in Section \ref{sec:MHMproof} we write the proof of Theorem \ref{thm:stcostintro} only for instanton moduli space.
\end{remark} 

\begin{remark}
Theorem \ref{thm:stcostintro} can be regarded as an equivariant, non-compact instance of the following results that hold in the setting of projective varieties. A result \cite[Corollary 6.29]{Bat} of Batyrev asserts that if two $K$-trivial smooth projective varieties $X$ and $X'$ are birationally equivalent, then one has an equality $$h^q(X,\Omega^p_X)=h^q(X',\Omega^p_{X'})$$ of Hodge numbers for any $p$ and $q$. 

Another analogous result in the hyperk\"{a}hler setting is due to Huybrechts, who shows in \cite[Theorem 4.6, Corollary 4.7]{Huy} that if two hyperk\"{a}hler manifolds $Y$ and $Y'$ are birationally equivalent, then they are deformation equivalent. So, for any $p$ and $q$ there is an equality of Hodge numbers  $$h^q(Y,\Omega^p_Y)=h^q(Y',\Omega^p_{Y'}).$$ 

It is therefore interesting to ask if the equality of $\chi_{-y}$-genera of Theorem \ref{thm:stcostintro} can be refined to an equality of the equivariant Hodge numbers $$h^q(M(r,n),\Omega^p_{M(r,n)} ),\   h^q(M^{c}(r,n), \Omega^p_{M^{c}(r,n)})\in \Q(t_1,t_2,e_0,\ldots,e_{r-1}).$$ (One strategy to produce such an equality of Hodge numbers would be to generalize the proof of Theorem \ref{thm:stcostintro} from Section \ref{sec:MHMproof} by lifting the maps ${\bf Gr}^F_{k}$ of Section \ref{subsec:MHMproperties} and the commutative diagram  \eqref{MHMNak} of Grothendieck groups to the level of derived categories. However, it is not known if such lifts exist in our setting. We thank J.~Sch\"{u}rmann for correspondence on this subject.)
\end{remark}

The remainder of this section is devoted to the proof of Theorem \ref{thm:stcostintro}. We abbreviate $M(r,n), M^c(r,n)$ and $M^0(r,n)$ by $M, M^c$ and $M^0$, respectively.

\subsection{Warm-up}\label{sec:warmup}
Our proof of Theorem \ref{thm:stcostintro} can be regarded as a generalization of the following argument for the case $k=0$. The argument is standard and can be found, for example, in \cite[Section 5.5]{Gin}. 

Let $q\colon M\to M^0$ and $q^c\colon M^c\to M^0$ denote the affinization maps. Both maps are $\mathbb{T}$-equivariant, proper and birational. As $M$ and $M^c$ are smooth, the Grauert-Riemenschneider theorem implies that for $i>0$ one has $$R^iq_*\omega_M=0, \ \ \ R^i{q^c_*}\omega_{M^c}=0,$$ 
where $\omega_{M}$, $\omega_{M^c}$ denote the dualizing line bundles.
As $M$ and $M^c$ are symplectic, it follows that for $i>0$ one has $$R^iq_*\mathcal{O}_M=R^i{q^c_*}\mathcal{O}_{M^c}=0.$$ As $M^0$ is the affinization of $M$ and $M^c$, we conclude that \begin{align}\label{k0} q_*[\mathcal{O}_M]={q^c_*}[\mathcal{O}_{M^c}]=\mathcal{O}_{M^0}\in K^{\mathbb{T}}_0(M^0).\end{align}

\subsection{Equivariant mixed Hodge modules}
To prove Theorem \ref{thm:stcostintro}, we imitate the argument of Section \ref{sec:warmup} in a suitable category. We consider the category of mixed Hodge modules, which is equipped with functors that can recover sheaves of differential forms from the trivial bundle. We thank Davesh Maulik for suggesting this approach, as well as Ben Davison and Lucien Hennecart for patiently sharing their expertise in the subject. 

For a smooth 
variety $X$ 
equipped with a $\mathbb{T}$-action, let $\MHM^{\T}(X)$ denote the category of $\mathbb{T}$-equivariant mixed Hodge modules.
A mixed Hodge module on $X$ consists of (1) a left $\mathscr{D}_X$-module $\mathscr{M}$, equipped with an increasing good filtration $F_{\bullet}$ compatible with the Bernstein filtration on $\mathscr{D}_X$, along with a (2) rational structure and (3) weight filtration, satisfying a list of compatiblities and conditions. To simplify the exposition, we omit notation for (2), (3) and these further conditions. In particular, the weight filtration in all our examples is straightforward; by \cite[(4.5.2)]{Sai} all mixed Hodge modules appearing in this section are pure. 
A $\TT$-equivariant mixed Hodge module consists of the above data, where all sheaves and morphisms are $\TT$-equivariant.

Let $D^{b,\T}_{\MHM}(X)$ denote Achar's equivariant derived category of mixed Hodge modules as constructed in \cite{A}. Introductions well-suited to our application may be found in \cite[Section 3]{DiMu}, \cite[Section 2]{MS}, and, for the non-equivariant case, in \cite[Section 2]{Fu}. We remark that $D^{b,\T}_{\MHM}(X)$ is in general different from the derived category of $\MHM^{\T}(X).$ Only the former carries the 6-functor formalism. However, both derived categories are triangulated categories equipped with bounded non-degenerate t-structures whose heart is $\MHM^{\T}(X)$. In particular there is a canonical isomorphism of K-groups $K_0(\MHM^\T(X))\cong K_0(D^{b,\T}_{\MHM}(X)).$ 

\subsubsection{}\label{subsec:MHMproperties} In Achar's construction, smoothness of $X$ is needed to ensure the existence of the 6-functor formalism. To associate a suitable category to the singular variety $M^0$, we use the following workaround explained in \cite[Section 2]{MS}. Given a normal and quasi-projective $\T$-variety $X'$, choose a $\T$-equivariant closed embedding $X'\hookrightarrow X$ into a smooth quasi-projective $\T$-variety $X$. The existence of such an $X$ and embedding follows from \cite[Theorem 1]{Sum} and \cite[Theorem 5.1.25]{CG}.

Then, let $\MHM^{\T}_{X'}(X)$ denote the abelian subcategory of $\MHM^\T(X)$ of equivariant mixed Hodge modules on $X$ whose support lies in $X'$. We summarize the relevant properties of $\MHM^{\T}(X)$ and $D^{b,\T}_{\MHM}(X).$ 

To a mixed Hodge module $\mathscr{M}\in \MHM^{\T}(X),$ one associates the de Rham complex $$\DR_X(\mathscr{M})=[0\to\mathscr{M}\to \Omega^{1}_X\otimes_{\O_X}\mathscr{M}  \to \cdots \to \Omega^{\dim(X)}_X\otimes_{\O_X}\mathscr{M}\to 0 ]$$ 
with nonzero terms in degrees $-n,\ldots,0$.  
This complex has a filtration $F_{\bullet}$ given by $$F_k \DR_X(\mathscr{M})=[0\to F_k\mathscr{M}\to\Omega^{1}_X \otimes_{\O_X}F_{k+1}\mathscr{M}  \to \cdots \to \Omega^{\dim(X)}_X\otimes_{\O_X}F_{k+\dim(X)}\mathscr{M}\to 0 ].$$ We let $Gr^{F}_k\DR_X(\mathscr{M})$ denote the associated graded complex $$[0\to Gr^F_k\mathscr{M}\to\Omega^{1}_X \otimes_{\O_X}Gr^F_{k+1}\mathscr{M}  \to \cdots \to \Omega^{\dim(X)}_X\otimes_{\O_X}Gr^F_{k+\dim(X)}\mathscr{M}\to 0 ]$$ 
of equivariant $\mathcal{O}_X$-modules with coherent cohomology.

For $k\in \Z$, there  is an induced transformation of Grothendieck groups $$\Gr^F_k\DR_X \colon K_0(\MHM^{\mathbb{T}}(X))\to K_0^\T(X).$$ We abbreviate this map by $\Gr^F_k$.

The maps $\Gr^F_k$ have the following properties.

\begin{enumerate}
\item\cite[p.~17]{MS} Let $X$ be a smooth quasiprojective $\T$-variety, $X'$ be any $\T$-variety and let $X'\hookrightarrow X$ be a $\T$-equivariant closed embedding. Given $[\mathscr{M}]\in K_0(\MHM^{\T}_{X'}(X))$ the image of $\Gr^F_{k}([\mathcal{M}])\in K^{\T}_0(X)$ belongs to the image of the pushforward map $$K_0^\T(X')\to K_0^\T(X).$$ 
\item \cite[Equation~(6)]{DiMu} Let $X$ be a smooth quasiprojective $\T$-variety and let $\mathcal{\Q}^{\vir}_X$ denote the constant Hodge module on $X$. The Hodge module $\mathcal{\Q}^{\vir}_X$ is the pure Hodge module associated to the $\mathscr{D}_X$-module $\mathcal{O}_X$ and perverse sheaf $\Q_X[\dim X]$ equipped with trivial filtrations. To be precise, the filtration $F$ on $\mathcal{O}_X$ is given by $F_{k}\mathcal{O}_X=\mathcal{O}_X$ when $k\geq 0$ and $F_{k}\mathcal{O}_X=0$ otherwise.

Then, for $0\leq k\leq \dim(X)$, one has $\Gr^F_{-k}(\mathcal{\Q}^{\vir}_X)=\Omega^{k}_X[-k].$ In particular, on the level of Grothendieck groups one has $$\Gr^F_{-k}([\mathcal{\Q}^{\vir}_X])=(-1)^k [\Omega^k_X]\in K^{\T}_0(X).$$

\item \cite[Proposition~2.4]{MS}, \cite[Section~3]{DiMu} Let $f\colon X\to Y$ be a proper morphism of smooth $\T$-varieties. Then for any $k\in \Z$, the following diagram commutes $$\xymatrix{ K_0(\MHM^\T(X))\ar[r]^{f_*}\ar[d]^{\Gr^{F}_{k} } &K_0(\MHM^\T(Y))\ar[d]^{\Gr^{F}_{k} }  \\ K^\T_0(X) \ar[r]^{f_*}  & K^\T_0(Y).}$$ We record the following generalization to handle the case when the target is singular. Let $f\colon X\to Y'$ be a morphism of $\T$-varieties such that $X$ is smooth, and let $Y'\hookrightarrow Y$ be a $\T$-equivariant closed embedding of $Y'$ in a smooth $\T$-variety $Y$. Then there is a commutative diagram \begin{equation}\label{MHMsquare} \xymatrix{ K_0(\MHM^\T(X))\ar[r]^{f_*}\ar[d]^{\Gr^{F}_{k} } &K_0(\MHM^{\T}_{Y'}(Y))\ar[d]^{\Gr^{F}_{k} }  \\ K^\T_0(X) \ar[r]^{f_*}  & K^\T_0(Y').}\end{equation}
In particular, the composition $\Gr^F_k\circ f_*$ is independent of the choice of embedding $Y'\hookrightarrow Y.$

\end{enumerate}

\subsection{Proof of Theorem \ref{thm:stcostintro}}\label{sec:MHMproof}
Let $q \colon M\to M^0$ and $q^c \colon M^c\to M^0$ denote the semisimplification maps.\footnote{For general smooth quiver varieties, these maps may not coincide with the affinization maps from Section \ref{sec:warmup}.} The maps $q$ and $q^c$ are $\mathbb{T}$-equivariant. It suffices to show that $$q_*[\Omega^{k}_M]=q^c_*[\Omega^{k}_{M^c}]\in K^{\T}_0(M^0)$$ for all $k\geq 0$.

Fix a $\T$-equivariant closed embedding $M^0\hookrightarrow V$ of the affine variety $M^0$ in a $\T$-equivariant smooth quasi-projective variety $V$ and let $0\leq k\leq 2rn$. By (\ref{MHMsquare}), there are commutative diagrams \begin{equation}\label{MHMNak} \xymatrix{ K_0(\MHM^\T(M))\ar[r]^{q_*}\ar[d]^{\Gr^{F}_{-k} } &K_0(\MHM^{\T}_{M^0}(V))\ar[d]^{\Gr^{F}_{-k}} &K_0(\MHM^\T(M^c))\ar[l]_{q^c_*}\ar[d]^{\Gr^{F}_{-k} }    \\ K^\T_0(M) \ar[r]^{q_*}  & K^\T_0(M^0) & \ar[l]_{q^c_*} K^\T_0(M^c).}\end{equation}

We trace $\Q_M^{\vir}\in \MHM^{\T}(M)$ and $\Q_{M^c}^{\vir}\in \MHM^{\T}(M^c)$ through the diagram. In one direction, we find $$q_*(\Gr^F_{-k} [\Q^{\vir}_M])= (-1)^p q_*[\Omega^{k}_M],\ \ q^c_*(\Gr^F_{-k} [\Q^{\vir}_{M^c}])= (-1)^p q^c_*[\Omega^{k}_{M^c}].$$

In the other direction, we explain in Section \ref{fineprint} how the proof of \cite[Proposition 10.3]{DHSM} implies that \begin{align}\label{dhsm}q_*\Q^{\vir}_M=\mathcal{BPS}_{\widehat{A}_0}\otimes\mathbb{L}^{rn}=q^c_*\Q^{\vir}_{M^c}\in \MHM^{\T}_{M^0}(V),\end{align} where $\mathcal{BPS}_{\widehat{A}_0}$ is an explicit mixed Hodge module associated to the preprojective algebra for the Jordan quiver defined in \cite[Definition 7.15]{DHSM}, and $\mathbb{L}$ is the Tate twist $H^*_c(\mathbb{A}^1,\Q)$.

We obtain a chain of equalities $$ q_*(\Gr^F_{-k} [\Q^{\vir}_M])=\Gr^F_{-k} q_* [\Q^{\vir}_M]=\Gr^F_{-k} q^c_* [\Q^{\vir}_{M^c}]=q^c_*(\Gr^F_{-k} [\Q^{\vir}_{M^c}]),$$ proving Theorem \ref{thm:stcostintro}.

\begin{remark} The equality (\ref{dhsm}) can be compared to (\ref{k0}). Note that the maps $q$ and $q^c$ are semi-small but not small in general. Given a small resolution $\rho \colon X\to Y$, by \cite[Corollary 5.6]{CDK} 
one has \begin{align}\label{small} \rho_* \Q^{\vir}_{X}=\mathcal{IC}_Y,\end{align} where $\mathcal{IC}$ denotes the intersection complex. 
In this sense, the equality (\ref{dhsm}) may be regarded as a replacement of (\ref{small}) valid for the semi-small resolutions $q$ and $q^c$. We expect that that there is a proof of \eqref{dhsm} using the one-parameter deformations of the quiver varieties $M$ and $M^c$.  These deformations are small resolutions of their 
semisimplifications; see for example \cite[Section 5.1]{N}. 
\end{remark}

\subsubsection{Fine print}\label{fineprint}
The computation used to prove \cite[Proposition 10.3]{DHSM} obtains the equality (\ref{dhsm}) as a composition of three equalities obtained from \cite[Propositions 8.4, 10.1, 10.2]{DHSM}. This argument as formulated in \cite{DHSM} does not immediately imply the equality (\ref{dhsm}) for two reasons: the result is stated only on the level of perverse sheaves, and stated only for one choice of stability condition on a framed quiver corresponding to $M$. In this section, we explain why the isomorphisms used in the proof of the proposition lift to the setting of $\mathbb{T}$-equivariant mixed Hodge modules and also hold for any generic choice of King stability condition.

First, the morphism used to prove \cite[Proposition 10.2]{DHSM} is already constructed on the level of mixed Hodge modules in \cite[Proposition 6.3]{D}. The argument of \cite[Proposition 6.3]{D} also also 
holds for the opposite choice $(0,\ldots,0,-1)$ of stability condition, and more generally for any choice of generic King stability condition. In particular, the relevant isomorphism of perverse sheaves is constructed in \cite[Theorem 6.4]{D2} for any King stability condition.

Second, the isomorphisms \cite[Lemmas 4.7, 4.10]{T} used in the equality \cite[Proposition 10.1]{DHSM} are constructed using the six-functor formalism, the vanishing cycles functor, and cohomology with respect to the perverse t-structure on the category of perverse sheaves. Each of these functors lifts to the derived category of mixed Hodge modules. We remark that \cite[Proposition 10.1]{DHSM} is also applied on the level of mixed Hodge modules in the proof of \cite[Proposition 4.5]{D}. Moreover, \cite[Lemmas 4.7, 4.10]{T} hold for any stability condition.

Finally, the two morphisms used to produce the morphism of \cite[Proposition 8.4]{DHSM} are constructed in \cite[Theorem 1.1]{DHSM} and \cite[Formula 65, Section 7]{D} on the level of mixed Hodge modules.

We conclude that the isomorphisms of perverse sheaves in \cite[Proposition 10.3]{DHSM} lift to morphisms of mixed Hodge modules, and that such an isomorphism also exists for the stability condition corresponding to $M^c$. As the rationalization functor taking a mixed Hodge module to its underlying perverse sheaf is conservative (see, for example, the discusion before \cite[Theorem 0.1]{Sai}), it follows that these morphisms of mixed Hodge modules must be isomorphisms.

Moreover, by \cite[Theorem 5.2]{A}, the forgetful functor from the category of  $\mathbb{T}$-equivariant mixed Hodge modules to the non-equivariant category is fully faithful, so the isomorphisms of \cite[Proposition 10.3]{DHSM} must also be isomorphisms in the equivariant category. See also the discussion in \cite[Section 11.5-6]{DHSM}, which explains how the  relevant maps and constructions appearing in \cite[Proposition 10.3]{DHSM} can be made $\T$-equivariant.

\vspace{.4in}

\noindent N. Arbesfeld, \texttt{noah.arbesfeld@univie.ac.at},\\
University of Vienna, Faculty of Mathematics \\

\noindent M. Kool, \texttt{m.kool1@uu.nl},\\
Utrecht University, Mathematical Institute \\

\noindent T. Laarakker, \texttt{tieslaarakker@gmail.com}
\end{document}